\documentclass[11pt,a4paper]{article}

\usepackage{latexsym,amsfonts,amsmath,amsthm,graphics}


\usepackage{graphicx}       
\usepackage[latin1]{inputenc}
\usepackage[T1]{fontenc}
\usepackage{microtype} 

\usepackage{cite}
\usepackage{array,colortbl}
\usepackage{amsmath,amsfonts,amssymb,bm,amsthm,mathtools,mathrsfs}
\DeclareFontFamily{U}{mathx}{\hyphenchar\font45}
\DeclareFontShape{U}{mathx}{m}{n}{
	<5> <6> <7> <8> <9> <10>
	<10.95> <12> <14.4> <17.28> <20.74> <24.88>
	mathx10
}{}
\DeclareSymbolFont{mathx}{U}{mathx}{m}{n}
\DeclareFontSubstitution{U}{mathx}{m}{n}
\DeclareMathAccent{\widecheck}{0}{mathx}{"71}

\def\citep#1#2{\cite[{#1}]{#2}}

\usepackage{algorithm}
\usepackage{algorithmic}

\newcommand{\C}{{\mathbb{C}}} 
\newcommand{\F}{{\mathbb{F}}} 
\newcommand{\N}{{\mathbb{N}}} 
\newcommand{\R}{{\mathbb{R}}} 
\newcommand{\Z}{{\mathbb{Z}}} 

\DeclareSymbolFont{bbold}{U}{bbold}{m}{n}
\DeclareSymbolFontAlphabet{\mathbbold}{bbold}

\newcommand{\bsg}{{\boldsymbol{g}}}

\newcommand{\bsk}{{\boldsymbol{k}}}

\newcommand{\bsq}{{\boldsymbol{q}}}

\newcommand{\bsu}{{\boldsymbol{u}}}
\newcommand{\bsv}{{\boldsymbol{v}}}

\newcommand{\bsx}{{\boldsymbol{x}}}

\newcommand{\bszero}{{\boldsymbol{0}}} 
\newcommand{\bsone}{{\boldsymbol{1}}}  

\newcommand{\bsgamma}{{\boldsymbol{\gamma}}}

\newcommand{\bseta}{{\boldsymbol{\eta}}}

\newcommand{\calD}{{\mathcal{D}}}

\newcommand{\calO}{{\mathcal{O}}}

\newcommand{\setu}{{\mathfrak{u}}}
\newcommand{\setv}{{\mathfrak{v}}}

\newcommand{\rme}{{\mathrm{e}}}

\newcommand{\mcol}{{\mathpunct{:}}}
\newcommand{\wal}{{\rm wal}}
\newcommand{\icomp}{\mathrm{i}}
	
\newcommand{\abs}[1]{\left\vert#1\right\vert}
\newcommand{\norm}[1]{\left\Vert#1\right\Vert}
\newcommand{\floor}[1]{\left\lfloor #1 \right\rfloor} 
\newcommand{\rd}{\,\mathrm{d}} 

\providecommand{\argmin}{\operatorname*{argmin}}

\newcommand{\supp}{\operatorname{supp}}
\newcommand{\tr}{{\rm tr}}
\newcommand{\tpmod}[1]{{\;(\operatorname{mod}\;#1)}}

\usepackage[hidelinks,hypertexnames=false,hyperindex=true,pdfpagelabels,linktoc=all]{hyperref}
\usepackage{bookmark}
\makeatletter 
\providecommand*{\toclevel@author}{999}
\providecommand*{\toclevel@title}{0}
\makeatother

\usepackage{enumitem}

\theoremstyle{plain}
\newtheorem{theorem}{Theorem}
\newtheorem{proposition}{Proposition}
\newtheorem{lemma}{Lemma}
\newtheorem{corollary}{Corollary}

\theoremstyle{definition}
\newtheorem{definition}{Definition}

\theoremstyle{remark}
\newtheorem{remark}{Remark}

\usepackage[T1]{fontenc}
\usepackage{lmodern}
\usepackage[a4paper]{geometry}
\usepackage{inputenc}
\usepackage{amsmath,amsfonts,amssymb}
\usepackage{bm}
\usepackage{mathtools}
\usepackage{hyperref}
\usepackage{tikz}
\usepackage{caption}
\usepackage{listings}
\usepackage{float}
\usepackage{pgfplots}
\usepackage{subcaption}
\usepackage{enumitem}
\usepackage{xcolor}

\usepackage[normalem]{ulem} 

\usepackage{booktabs}
\usepackage{multirow}

\allowdisplaybreaks

\usepackage{changebar}

\pgfplotsset{every tick label/.append style={font=\scriptsize}}
\newenvironment{customlegend}[1][]{
	\begingroup
	\csname pgfplots@init@cleared@structures\endcsname
	\pgfplotsset{#1}
}{
	\csname pgfplots@createlegend\endcsname
	\endgroup
}

\def\addlegendimage{\csname pgfplots@addlegendimage\endcsname}
\pgfplotsset{
	cycle list={%
		{draw=black,mark=star,solid},
		{draw=black, mark=square,solid},
		{draw=black,mark=+,solid},
		{black,mark=o},
		{draw=black, mark=none,solid}}
}

\def\addlegendimage{\csname pgfplots@addlegendimage\endcsname}
\pgfplotsset{
	cycle list={%
		{draw=black,mark=star,solid},
		{draw=black, mark=square,solid},
		{draw=black,mark=+,solid},
		{black,mark=o},
		{draw=black, mark=none,solid}}
}

\definecolor{mycolor1}{rgb}{0.48500,0.70000,1.00000}
\definecolor{mycolor2}{rgb}{0.00000,0.20000,0.50000}
\definecolor{mycolor-alpha1}{rgb}{0.30000,0.45000,0.85000}
\definecolor{mycolor-alpha2}{rgb}{0.05000,0.65000,0.20000}
\definecolor{mycolor-alpha3}{rgb}{0.64000,0.22500,0.75000}

\definecolor{mycolor3}{rgb}{0.10000,0.80000,0.15000}
\definecolor{mycolor4}{rgb}{0.00000,0.50000,0.25000}
\definecolor{mycolor5}{rgb}{0.84000,0.29000,1.00000}
\definecolor{mycolor6}{rgb}{0.44000,0.16500,0.50000}
\definecolor{mycolor2-time}{rgb}{0.85000,0.32500,0.09800}%

\begin{document}

\title{Construction of good polynomial lattice rules in weighted Walsh spaces by an alternative component-by-component construction}

\author{Adrian Ebert\thanks{A.~Ebert, P.~Kritzer, and O.~Osisiogu are supported by the Austrian Science Fund (FWF): Project F5506, 
which is part of the Special Research Program ``Quasi-Monte Carlo Methods: Theory and Applications''.}, 
Peter Kritzer, Onyekachi Osisiogu, Tetiana Stepaniuk\thanks{T.~Stepaniuk is supported by the Alexander von Humboldt Foundation.}}

\date{\today}

\maketitle

\begin{abstract}
We study the efficient construction of good polynomial lattice rules, which are special instances of quasi-Monte Carlo (QMC) methods. 
The integration rules obtained are of particular interest for the approximation of multivariate integrals in weighted Walsh spaces. 
In particular, we study a construction algorithm which assembles the components of the generating vector, which is in this case a vector of polynomials over a finite field, of the polynomial lattice rule in a component-wise fashion. We show that the constructed QMC rules achieve the almost optimal error convergence order in the function spaces under consideration and prove that the obtained error bounds can, under certain conditions on the involved weights, be made independent of the dimension. We also demonstrate that our alternative component-by-component construction, which is independent of the underlying smoothness of the function space, can be implemented relatively easily in a fast manner.
Numerical experiments confirm our theoretical findings.
\end{abstract}



\section{Introduction} \label{sec:intro}

We are interested in studying multivariate numerical integration, more precisely, we consider numerical integration of a sub-class of 
square-integrable functions $f \in L^2([0,1]^d)$. For the approximation of the $d$-dimensional integrals we will use quasi-Monte Carlo (QMC) rules, which are equal-weight quadrature/cubature rules, that is, 
\begin{equation*}
	I_d(f)
	:= 
	\int_{[0,1]^d} f(\bsx) \rd \bsx 
	\approx 
	Q_{N,d}(f)
	:=
	\frac1N 
	\sum_{n=0}^{N-1} f(\bsx_n)
	,
\end{equation*}
where the quadrature/cubature points $\bsx_0, \ldots, \bsx_{N-1} \in [0,1]^d$ are chosen in a deterministic way. This is in contrast to Monte Carlo rules, which are of the same form as QMC methods but are based on randomly chosen integration nodes. With QMC rules, we try to make a deliberate and sophisticated choice of the points $\bsx_n$ with the aim of obtaining better error bounds than for Monte Carlo methods. Here, the key challenge is to systematically devise integration nodes that yield good approximation results simultaneously for a wide class of integrand functions that may depend on a large number of variables. Additionally, in order to obtain a low approximation error, we may need to construct millions of good integration nodes in very high dimensions, which presents a considerable computational challenge.  

There are two main families of point sets for QMC methods which are commonly considered in the literature. These are, on the one hand, lattice point sets, as introduced independently by Korobov (see \cite{Kor59}) and Hlawka (see \cite{H62}). For more recent introductions to lattice rules, we refer to \cite{Nied92, SJ94}. The other class of commonly used QMC integration node sets is that of (digital) $(t,m,d)$-nets and $(t,d)$-sequences, as introduced by Niederreiter, building up on ideas by Sobol' and Faure (see \cite{N87,Nied92}). In this article, we will consider a special instance of $(t,m,d)$-nets, namely so-called polynomial lattice point sets. Originally introduced in \cite{N92a}, polynomial lattices owe their name to their construction principle which resembles that of (ordinary) lattice point sets. While the construction of lattice point sets is based on integer arithmetic, polynomial lattice point sets are based on polynomial arithmetic over finite fields $\F_b$ with $b$ elements, where $b$ is prime. In particular, a polynomial lattice point set consists of $b^m$ points in $[0,1]^d$ that are constructed by means of a modulus $p\in \F_b [x]$ with $\deg(p)=m$, and a generating vector $\bsg\in (\F_b [x])^d$ (we refer to Section \ref{sec:poly_def} for the precise definition). The resulting QMC rule using the polynomial lattice point set as integration nodes is then called a polynomial lattice rule. 
In this article, we will use irreducible polynomials as the moduli of polynomial lattice rules. 

We remark that not every choice of a generating vector $\bsg$ yields a polynomial lattice rule with good approximation properties. In general, it is usually highly non-trivial to find good generating vectors of polynomial lattice rules, and there are (except for special cases) no explicit constructions of such good generating vectors known. Hence, one has to resort to computer search algorithms for finding generating vectors of polynomial lattice point sets with high quality. In this article, we study the worst-case setting, that is, we consider a particular normed function space and analyze the so-called worst-case error, which is the integration error of the considered QMC methods in the supremum over the unit ball of the space. The worst-case error will then serve as our quality measure for the constructed polynomial lattice rules.

It is known that polynomial lattice rules are well suited for the numerical integration of functions that can be represented by Walsh series (cf. \cite{DKPS05,DP05,DP10}). Throughout this article, we will therefore consider a ``weighted'' function space (in the sense of Sloan and Wo\'{z}niakowski (cf. \cite{SW98})) which consists of elements whose Walsh coefficients decay sufficiently fast. The prescribed decay will be characterized by a smoothness parameter $\alpha >1$ (in some publications this parameter is also referred to as the ``digital smoothness parameter'' in the context of Walsh series). Indeed, the parameter $\alpha$ is linked to the speed of decay of the Walsh coefficients of the functions in our space, but there is also a connection to the number of derivatives that exist for the elements of the space (we refer to \cite{DP10} and the references therein for details). Furthermore, the varying importance of different subsets of variables will be modeled by a collection of positive reals $\bsgamma = (\gamma_\setu)_{\setu \subseteq \{1,\ldots,d\}}$ which we refer to as weights. As pointed out in \cite{SW98} and numerous other papers, this concept is justified by practical high-dimensional problems in which different coordinates may indeed have a very different degree of influence on the values of an integral. The weights will be incorporated in the inner product and norm of the function space in a suitable way. Using this setting, it is plausible that a nominally very high-dimensional problem may have a rather low ``effective dimension'', i.e., only a certain, possibly small, part of the components has a significant influence on the integration problem and the error made by approximative algorithms.

The function space studied here is closely related to other function spaces considered in the literature, such as in \cite{DKPS05, DP05, DP10}, and results shown for the function space considered in the present paper immediately yield corresponding results for some of the Walsh spaces considered in these references. We refer to Section \ref{sec:walsh-polylat} below for further details.

The first efficient construction algorithm for good generating vectors of polynomial lattice point sets was introduced in \cite{DKPS05}, where a so-called component-by-component (CBC) approach was formulated and analyzed. Component-by-component constructions are greedy algorithms which construct the generating vector one component at a time. In the context of QMC integration, CBC algorithms were initially considered for (ordinary) lattice point sets, with the first examples in the literature going back to Korobov (cf. \cite{Kor63}), and later also Sloan and Reztsov (cf. \cite{SR02}). The CBC construction was then made widely applicable by the formulation of a fast construction algorithm. This fast CBC construction, which is due to Nuyens and Cools (see, e.g., \cite{N14, NC06a, NC06b}), makes the CBC algorithm computationally competitive and is currently the standard method to construct high-dimensional lattice point sets with good quality. It is well known (see, e.g., \cite{DKPS05} and again \cite{N14}) that CBC constructions also work for the efficient search for generating vectors of polynomial lattice point sets; and also in this case a fast algorithm is available.

In the present paper, an alternative CBC algorithm for the efficient construction of generating vectors of polynomial lattice point sets will be formulated. Opposed to the standard CBC construction that uses the worst-case error of the considered function space as 
the quality measure, our algorithm is based on an alternative quality criterion, which is in particular independent of the parameter $\alpha$. 
We stress that therefore no prior knowledge of the smoothness parameter $\alpha$ is required to construct the generating vector $\bsg$. 
The resulting generating vector will still achieve the almost optimal rate of convergence, for arbitrary values of the smoothness parameter $\alpha>1$, and this result can be stated independently of the dimension $d$, assuming that the weights satisfy certain summability conditions. The standard CBC algorithms construct the generating vector specifically with the smoothness $\alpha$ as an input parameter. We see the independence of $\alpha$ in the construction algorithm presented in this paper as a big advantage of our new method.

The rest of the article is structured as follows. In Section~\ref{sec:walsh-polylat}, we define the concept of polynomial lattice rules, introduce the weighted function space under consideration, and analyze the corresponding worst-case error expression. Furthermore, we introduce the quality criterion on which our algorithm is based and show the existence of good polynomial lattice rules for our setting. In Section~\ref{sec:CBC}, a novel variant of the CBC construction algorithm for good polynomial lattice rules is proposed. We prove that the resulting polynomial lattice rules achieve the almost optimal order of convergence, with error bounds independent of the dimension assuming certain conditions on the weights are satisfied. In Section~\ref{sec:fast_impl}, we show that the introduced construction method can be implemented in a fast manner which is competitive with the state-of-the-art component-by-component algorithm. Finally, the paper is concluded in Section~\ref{sec:num}, where we illustrate our main result by numerical experiments.

For the remainder of the article we fix some basic notation. We denote the set of positive integers by $\N$ and the set of non-negative integers by $\N_0$. To denote sets containing indices of components, we use fraktur font, e.g., $\setu \subset \N$, and additionally write $\{1 \mcol d\} := \{1,\ldots,d\}$ for short. For the projection of a vector $\bsx \in [0,1]^d$ or $\bsk \in \N_0^d$ onto the components with indices in a set $\setu \subseteq \{1 \mcol d\}$, we write $\bsx_\setu = (x_j)_{j \in \setu}$ or $\bsk_\setu = (k_j)_{j \in \setu}$, respectively. With a slight abuse of notation, we will frequently identify elements of the finite field $\F_b$ of prime cardinality $b$ with elements of the group of integers modulo $b$ denoted by $\Z_b$.

\section{Preliminaries} \label{sec:walsh-polylat}

In this section, we define polynomial lattice rules, introduce the function spaces under consideration, and then define the quality measure studied in this paper, for which we prove a first, non-constructive existence result.

\subsection{Polynomial lattice point sets} \label{sec:poly_def}

For a prime $b$, let $\F_{b}$ be the finite field with $b$ elements and let $\F_b ((x^{-1}))$ be the field of formal Laurent series over $\F_b$. 
Elements of $\F_b ((x^{-1}))$ have the form
\begin{equation*}
	L
	=
	\sum_{\ell=w}^\infty t_{\ell} \,x^{-\ell}
	,
\end{equation*}
where $w$ is an arbitrary integer and all $t_{\ell}\in\F_b$. Furthermore, we denote by $\F_b [x]$ the set of all polynomials over $\F_b$. For an integer $m \in \N$ let the map $v_m: \F_b ((x^{-1})) \to [0,1)$ be defined by
\begin{equation*}
	v_m \left(\sum_{\ell =w}^\infty t_\ell \, x^{-\ell}\right)
	=
	\sum_{\ell=\max (1,w)}^m t_{\ell} \, b^{-\ell}
\end{equation*}
and for a vector $\bsq=(q_1,\ldots,q_d) \in (\F_b ((x^{-1})))^d$ set $v_m(\bsq) = (v_m(q_1),\ldots,v_m(q_d)) \in [0,1)^d$.
Given an integer $n \in \N_0$ with $b$-adic expansion $n=n_0 + n_1 b + \cdots + n_a b^{a}$, we will frequently associate $n$ with the polynomial
\begin{equation*}
	n(x)
	:=
	\sum_{k=0}^a n_k \, x^k \in \F_b[x]
	.
\end{equation*} 

The definition of a polynomial lattice point set is as follows.

\begin{definition}[Polynomial lattice point set] \label{def:poly_lat}
	Let $b$ be prime and let $m,d \in \N$ be given. Furthermore, choose $p\in \F_b [x]$ with $\deg (p)=m$, and let 
	$g_1,\ldots,g_d \in \F_b [x]$. Then the point set $P(\bsg,p)$, defined as the collection of the $b^m$ points 
	\begin{equation*}
		\bsx_n
		:=
		\left(v_m \left( \frac{n(x)\, g_1(x)}{p(x)} \right),\ldots, v_m\left( \frac{n(x)\, g_d(x)}{p(x)} \right)\right) \in [0,1)^d
	\end{equation*}
	for $n \in \F_b[x]$ with $\deg(n)<m$, is called a \emph{polynomial lattice point set} and the vector of polynomials
	$\bsg=(g_1,\ldots,g_d) \in (\F_b[x])^d$ is called the generating vector.
\end{definition}

The point sets $P(\bsg,p)$ which are considered in this paper are often called polynomial lattices and a QMC rule using $P(\bsg,p)$ is referred 
to as a polynomial lattice rule. The polynomial $p$ is referred to as the modulus. Without loss of generality, we can restrict the choice of the components $g_j$ of the generating vector $\bsg$ to the sets $G_{b,m}$ or $G^\ast_{b,m}$, which are subsets of $\F_b [x]$, defined as 
\begin{align*}
	G_{b,m}
	:=
	\{ g\in\F_b [x] \mid \deg(g) < m \}
	\quad \text{and} \quad
	G^\ast_{b,m}
	:=
	\{ g\in\F_b [x] \setminus \{0\} \mid \deg(g) < m \}
	.
\end{align*}

For two arbitrary vectors $\bsu(x)=(u_1(x),...,u_d(x)), \bsv(x)=(v_1(x),...,v_d(x)) \in (\F_b[x])^d$ we define the inner product 
\begin{equation*}
	\bsu(x) \cdot \bsv(x)
	:=
	\sum_{j=1}^{d} u_j(x )v_j(x) \in \F_b[x]
	,
\end{equation*}
and we write $ v(x) \equiv 0 \tpmod {p(x)}$ if $p(x)$ divides $v(x)$ in $\F_b[x]$. For integer $k \in \N_0$ with $b$-adic expansion 
$k=\kappa_0 + \kappa_1 b + \cdots + \kappa_{a-1} b^{a-1}$, we define the truncation map $\tr_m: \N_0 \to G_{b,m}$ via
\begin{equation*}
	\tr_m(k)
	:=
	\kappa_0 + \kappa_1 x + \cdots + \kappa_{m-1} x^{m-1} 
\end{equation*}
with $d$-variate generalization $\tr_m(\bsk)$ applied component-wise and additionally introduce the vector 
$\widetilde{\tr}_m(\vec{k}) = (\kappa_0,\kappa_1,\ldots,\kappa_{m-1}) \in \F_b^m$. If $a<m$, then the $\kappa_i$ with $i>a$ are 
intepreted as being equal to zero.

Polynomial lattice point sets are special instances of so-called digital nets, which are constructed using linear algebra over finite fields or rings (see \cite{Nied92} for an introduction). Every digital net has a dual net, which plays a crucial role in the expression of the integration error of a QMC rule using the net. For a polynomial lattice point set with generating vector $\bsg$ and modulus $p$ with $\deg(p)=m$, its dual net (or, 
in this case, dual polynomial lattice) $\calD(\bsg,p)$ equals 
\begin{equation} \label{eq:dual}
	\calD(\bsg,p)
	=
	\{\bsk \in \N_0^d \mid \tr_m(\bsk) \cdot \bsg \equiv 0 \tpmod p \} 
	.
\end{equation} 

Furthermore, for a set $\setu \subseteq \{1 \mcol d\}$ we introduce the notation
\begin{equation*}
	\calD_\setu
	=
	\calD_\setu(\bsg, p)
	=
	\calD_\setu(\bsg_\setu)
	:=
	\{ \bsk_\setu  \in \N^{|\setu|} \mid \tr_m(\bsk_\setu) \cdot \bsg_\setu \equiv 0 \tpmod{p} \} 
	.
\end{equation*}

\subsection{Walsh series representation} \label{subsec:walsh}

We consider numerical integration for square-integrable functions $f \in L^2([0,1]^d)$ which can be represented in terms of their Walsh series. This particular series representation of a function is based on the so-called Walsh functions which are given in the following definition.

\begin{definition} \label{def:walsh_functions}
	Let $b \ge 2$ be an integer. For a non-negative integer $k$ with base $b$ expansion $k=\kappa_0 + \kappa_1 b + \cdots + \kappa_{a-1} b^{a-1}$, we define the $k$-th Walsh function $_b\wal_k :[0,1) \to \C$ in base $b$ by
	\begin{equation*}
		_b\wal_k (x)
		:=
		\rme^{2\pi\icomp (\kappa_0 \xi_1 + \kappa_1 \xi_2 + \cdots + \kappa_{a-1} \xi_{a})/b}
		,
	\end{equation*}
	where $x\in [0,1)$ is represented as $x=\xi_1 b^{-1} + \xi_2 b^{-2} + \cdots$ with coefficients $\xi_i \in \{0,1,\ldots,b-1\}$ 
	(unique in the sense that infinitely many of the $\xi_i$ must be different from $b-1$).

	For $d \in \N$, an integer vector $\bsk=(k_1,\ldots,k_d) \in \N_0^d$ and $\bsx=(x_1,\ldots,x_d)\in [0,1)^d$, we define 
	the $\bsk$-th ($d$-variate) Walsh function $_b\wal_{\bsk} :[0,1)^d \to \C$ in base $b$ by
	\begin{equation*}
		_b\wal_{\bsk} (\bsx)
		:=
		\prod_{j=1}^d \ _b\wal_{k_j} (x_j)
		.
	\end{equation*}
\end{definition}

In the following, we will consider the base $b \ge 2$ as fixed, and then simply write $\wal_k$ or $\wal_{\bsk}$ instead of $_b\wal_k$ or $_b\wal_{\bsk}$, respectively. Note that for any  function $f \in L^2([0,1]^d)$ the  Walsh series of $f$ is given by
\begin{equation} \label{eq:Walsh_series}
	f(\bsx)
	=
	\sum_{\bsk \in \N_0^d} \hat{f}(\bsk) \, \wal_\bsk(\bsx)
	,
\end{equation}
where the $\bsk$-th Walsh coefficient of $f$ is defined as
\begin{equation*}
	\hat{f}(\bsk)
	=
	\int_{[0,1]^d} f(\bsx) \, \overline{\wal_\bsk(\bsx)} \rd \bsx
	.
\end{equation*}

For the numerical integration of functions given in terms of their Walsh series as in \eqref{eq:Walsh_series}, it is common to consider quasi-Monte Carlo rules which are based on digital nets or their infinite counterparts, digital sequences. In this article, we will employ the special digital net type of polynomial lattice rules, as introduced in the previous section, in order to approximate integrals by QMC rules. It is well known, see, e.g., \cite[Theorem 6.4]{DKS13}, that approximating the integral $I_d(f)$ using a QMC rule based on the polynomial lattice $P(\bsg,p) = \{\bsx_0,\ldots,\bsx_{N-1}\}$ with generating vector $\bsg \in G_{b,m}^d$ and modulus $p \in \F_b[x]$ with $\deg(p)=m$, i.e.,
\begin{equation*}
	Q_{b^m,d}(f,\bsg)
	=
	Q_{b^m,d}(f;P(\bsg,p))
	:=
	\frac{1}{b^m} \sum_{n=0}^{b^m-1} f(\bsx_n)
	\approx
	\int_{[0,1]^d} f(\bsx) \rd \bsx 
	=:
	I_d(f)
	,
\end{equation*}
leads to an integration error of the form
\begin{equation} \label{eq:int_error}
	Q_{b^m,d}(f,\bsg) - I_d(f)
	=
	\sum_{\bszero \ne \bsk \in \calD(\bsg,p)} \hat{f}(\bsk)
\end{equation}
with $\calD(\bsg,p)$ as in \eqref{eq:dual}, provided that $f$ can be represented by a Walsh series as in \eqref{eq:Walsh_series}. In particular, we will require in this paper that $f$ lies in a function space referred to as a Walsh space, which we introduce next.

\subsection{The weighted Walsh space  \texorpdfstring{$W_{d,\bsgamma}^{\alpha}$}{TEXT}} \label{subsec:weight_wal}

In this section, we will introduce the function space under consideration in this paper, which consists of functions $f$ that admit a representation as in \eqref{eq:Walsh_series} and for which the Walsh coefficients $\hat{f}(\bsk)$ decay at a prescribed rate. For this purpose, we will first define a decay function as follows.

For prime base $b$ and a given real smoothness parameter $\alpha > 1$, we define the decay function $r_\alpha: \N_0 \to \R$ by
\begin{equation*}
	r_\alpha (k) = r_{\alpha}(b,k)
	:=
	\left\{\begin{array}{cc}
	1, & {\text{if }} k=0 , \\ 
	b^{\alpha \psi_b (k)}, & {\text{if }} k \ne 0 
	,
	\end{array}\right.
\end{equation*}
where we set $\psi_b(k):=\floor{\log_b(k)}$ for $k\in \N$. Additionally, we introduce the auxiliary quantity
\begin{equation*}
	\mu_b (\alpha)
	:=
	\sum_{k=1}^\infty r^{-1}_{\alpha} (k)= \sum_{a=0}^\infty \frac{1}{b^{a\alpha}} 
	\sum_{k=b^a}^{b^{a+1}-1} 1=\sum_{a=0}^\infty \frac{(b-1)b^a}{b^{a\alpha}}
	=
	\frac{b^{\alpha}(b-1)}{b^\alpha - b}
	.
\end{equation*} 
For an integer vector $\bsk=(k_1,\ldots,k_d) \in \N_0^d$ and positive weights $\bsgamma = (\gamma_\setu)_{\setu \subseteq \{1:d\}}$, 
we define the (weighted) multivariate generalization of $r_\alpha$ as
\begin{equation*}
	r_\alpha (\bsk)
	:=
	\prod_{j=1}^d r_\alpha (k_j) 
	\quad \text{and} \quad
	r_{\alpha,\bsgamma}(\bsk)
	:=
	\gamma_{\supp(\bsk)}^{-1} \, r_\alpha (\bsk) 
	=
	\gamma_{\supp(\bsk)}^{-1} \prod_{j \in \supp(\bsk)} b^{\alpha \psi_b(k_j)}
\end{equation*}
with $\supp(\bsk) := \{ j \in \{1 \mcol d\} \mid k_j \ne 0 \}$. As mentioned in the introduction, the weights $\gamma_{\setu}$, which are incorporated in the decay function (and thus in the definition of the function space), model the varying importance of subsets of variables $\bsx_{\setu} = (x_j)_{j \in \setu}$ on the integration problem. The weights will play a crucial role in our effort to overcome an exponential dependence on the integration error of the QMC rules on the dimension $d$. This exponential dependence is sometimes also referred to as the curse of dimensionality.

The weighted Walsh space $W_{d,\bsgamma}^{\alpha}$ is then defined as the space of all functions 
$f \in L^2([0,1]^d)$ for which the norm $\norm{f}_{W_{d,\bsgamma}^{\alpha}}$, given by
\begin{equation*}
	\norm{f}_{W_{d,\bsgamma}^{\alpha}}
	:=
	\sup_{\bsk \in \N_0^d} |\hat{f}(\bsk)| \, r_{\alpha,\bsgamma}(\bsk)
\end{equation*}
is finite, that is,
\begin{equation*}
	W_{d,\bsgamma}^{\alpha}
	:=
	\{f \in L^2([0,1]^d) \mid \norm{f}_{W_{d,\bsgamma}^{\alpha}} < \infty \}
	,
\end{equation*}
where $\alpha>1$, and where the weights $\bsgamma = (\gamma_\setu)_{\setu \subseteq \{1:d\}}$ are strictly positive. As pointed out in the introduction, the function space studied here is closely related to other Walsh spaces considered in the literature, such as in \cite{DKPS05, DP05, DP10}. To be more precise, the spaces considered in these references have the norm given by $\norm{f}^2 =\sum_{\bsk \in \N_0^d} |\hat{f}(\bsk)|^2 \, r_{\alpha,\bsgamma}(\bsk)$. It can be shown easily that the worst-case error in these spaces is exactly the square root of the worst-case error in $W_{d,\bsgamma}^{\alpha}$. Therefore the results shown here immediately imply results for the spaces considered in references like \cite{DKPS05, DP05, DP10}.

In this article, we will use the so-called worst-case error to assess the quality of integration rules. 
For the function space $W_{d,\bsgamma}^{\alpha}$, the worst-case error of a QMC rule $Q_{N,d}(\cdot; P)$ 
with underlying point set $P \subseteq [0,1]^d$ consisting of $N$ points,
is defined as
\begin{equation*}
	e_{N,d,\alpha,\bsgamma}(P)
	:=
	\sup_{\substack{f \in W_{d,\bsgamma}^{\alpha} \\ \norm{f}_{W_{d,\bsgamma}^{\alpha} \le 1}}} \abs{Q_{N,d}(f; P) - I_d(f)}
	.
\end{equation*}
As was shown in \cite{EKOS2020}, the worst-case error in $W_{d,\bsgamma}^{\alpha}$ of a polynomial lattice rule  
takes an explicit form.    

\begin{theorem}[Worst-case error expression] \label{thm:wce_dig_net}
	Let $m,d \in \N$, a real $\alpha > 1$, prime $b$, and positive weights $\bsgamma = (\gamma_{\setu})_{\setu \subseteq \{1:d\}}$ be given. 
	Then the worst-case error $e_{b^m,d,\alpha,\bsgamma}(P(\bsg,p))$ of a QMC rule  based on the polynomial lattice $P(\bsg,p)$
	with generating vector $\bsg \in (\F_b[x])^d$ and modulus $p \in \F_b[x]$ with $\deg(p)=m$ in the space $W_{d,\bsgamma}^{\alpha}$ satisfies
	\begin{equation*}
		e_{b^m,d,\alpha,\bsgamma}(P(\bsg,p))
		=
		\sum_{\bszero \ne \bsk \in \calD(\bsg,p)} (r_{\alpha,\bsgamma}(\bsk))^{-1}
		,
	\end{equation*}
    where $\calD(\bsg,p)$ is defined as in \eqref{eq:dual}.
\end{theorem}
Henceforth, we will denote the worst-case error $e_{b^m,d,\alpha,\bsgamma}(P(\bsg_,p))$ of a QMC rule that is based on the polynomial lattice
$P(\bsg,p)$ in the space $W_{d,\bsgamma}^\alpha$ simply by $e_{b^m,d,\alpha,\bsgamma}(\bsg)$. 

The integration error for a single $f\in W_{d,\bsgamma}^{\alpha}$ can then be directly related to the worst-case error of a polynomial lattice rule. Starting from the expression in \eqref{eq:int_error}, an application of H\"older's inequality with parameters 1 and $\infty$ yields the estimate
\begin{align*}
	\abs{Q_{b^m,d}(f;P(\bsg,p)) - I_d(f)}
	&=
	\abs{\sum_{\bszero \ne \bsk \in \calD(\bsg,p)} \hat{f}(\bsk)}
	=
	\abs{\sum_{\bszero \ne \bsk \in \N_0^d} \hat{f}(\bsk) \, r_{\alpha,\bsgamma}(\bsk) 
	\, (r_{\alpha,\bsgamma}(\bsk))^{-1} \, \bsone_{\calD(\bsg,p)}(\bsk)} \nonumber \\
	&\le
	\left( \sup_{\bsk \in \N_0^d} |\hat{f}(\bsk)| \, r_{\alpha,\bsgamma}(\bsk) \right) 
	\left( \sum_{\bszero \ne \bsk \in \calD(\bsg,p)} (r_{\alpha,\bsgamma}(\bsk))^{-1} \right)
	\\
	&=
	\norm{f}_{W_{d,\bsgamma}^{\alpha}} e_{b^m,d,\alpha,\bsgamma}(\bsg)
\end{align*}
with $\bsone_{\calD(\bsg,p)}$ denoting the indicator function of the dual net $\calD(\bsg,p)$. Therefore, it becomes evident that we need to construct polynomial lattice rules such that the associated worst-case error $e_{b^m,d,\alpha,\bsgamma}(\bsg)$ is as small as possible.

\subsection{The existence of good polynomial lattice rules}

In this section, we study the existence of polynomial lattice rules with a small worst-case error which satisfies a certain asymptotic decay with respect to the number of cubature nodes. Here, we introduce an alternative quality measure which, opposed to the worst-case error expression obtained in Theorem~\ref{thm:wce_dig_net}, is independent of the parameter $\alpha$. To this end, we extend the definition of the function $r_{\alpha,\bsgamma}$ to the case where we allow $\alpha=1$, i.e.,
\begin{equation*}
	r_{1,\bsgamma}(\bsk)
	:=
	\gamma_{\supp(\bsk)}^{-1} \prod_{j \in \supp(\bsk)} b^{\floor{\log_b(k_j)} }\quad \text{for} \quad \bsk\in\N_0^d
	.
\end{equation*}

\noindent
For given $\alpha \ge 1$, given positive weights $\bsgamma = (\gamma_{\setu})_{\setu \subseteq \{1:d\}}$, modulus $p \in \F_b[x]$ with $\deg(p)=m$, and $\bsg \in (\F_b[x])^d$, we introduce the quantities
\begin{equation} \label{eq:qual_meas}
	T_{\bsgamma}(\bsg,p)
	:= 
	\sum_{\bszero \ne \bsk \in A_p(\bsg)} (r_{1,\bsgamma}(\bsk))^{-1} 
	,\ \mbox{and} \qquad
	T_{\alpha,\bsgamma}(\bsg,p)
	:= 
	\sum_{\bszero \ne \bsk \in A_p(\bsg)} (r_{\alpha,\bsgamma}(\bsk))^{-1}
\end{equation}      
with index set $A_p(\bsg)$ given by
\begin{equation*}
	A_p(\bsg)
	:=
	\{ \bsk \in \{0,1,\ldots,b^m-1\}^d \mid \bsk \in \calD(\bsg,p) \}.
\end{equation*}
Furthermore, for a set $\emptyset\neq\setu \subseteq \{1 \mcol d\}$, we introduce the sets
\begin{align*}
	A_\setu 
	&
	=
	A_{p,\setu}(\bsg)
	=
	A_{p,\setu}(\bsg_\setu)
	:=
	\{ \bsk_\setu \in \{0,1,\ldots,b^m-1\}^{\abs{\setu}} \mid \bsk_\setu \in \calD_{\setu}(\bsg,p) \}
	, \\
	A^\ast_\setu 
	&
	=
	A^\ast_{p,\setu}(\bsg)	
	=
	A^\ast_{p,\setu}(\bsg_\setu)
	:=
	\{ \bsk_\setu \in \{1,\ldots,b^m-1\}^{\abs{\setu}} \mid \bsk_\setu \in \calD_{\setu}(\bsg,p) \},
\end{align*}
and for the modulus $p \in \F_b[x]$ we define the indicator function $\delta_p: \F_b[x] \to \{0,1\}$ by
\begin{equation*}
	\delta_p(q)
	:=
	\begin{cases}
		1 & \text{if } q \equiv 0 \tpmod{p}, \\
		0 & \text{if } q \not\equiv 0 \tpmod{p}.
	\end{cases}
\end{equation*}

We can then prove the following existence result for good generating vectors $\bsg$ with respect to the quality measure $T_{\bsgamma}(\bsg,p)$.

\begin{theorem}[Existence result w.r.t. $T_{\bsgamma}$] \label{thm:exist_T}
	Let $\bsgamma = (\gamma_{\setu})_{\setu \subseteq \{1:d\}}$ be positive weights. For every irreducible $p \in \F_b[x]$ of degree $m \in \N$ 
	there exists a generating vector $\bsg \in G_{b,m}^d$ such that
	\begin{align*}
		T_{\bsgamma}(\bsg,p) 
		= 
		\sum_{\bszero \ne \bsk \in \{0,1,\ldots,b^m-1\}^d} \frac{\delta_p(\tr_m(\bsk) \cdot \bsg)}{r_{1,\bsgamma}(\bsk)} 
		\le
		\frac{1}{b^m} \sum_{\emptyset\neq\setu\subseteq \{1:d\}} \gamma_\setu \left( m \, (b-1) \right)^{|\setu|} 
		.
	\end{align*}
\end{theorem}

\begin{proof}
	Let $\bszero \ne \bsk \in \N_0^d$ be such that there exists $j \in \{1,\ldots,d\}$ with $\tr_m(k_j) \ne 0$, and assume without loss of generality that $j = d$. Then we find that
	\begin{equation} \label{eq:ind_func}
		\sum_{\bsg \in G_{b,m}^d} \delta_p(\tr_m(\bsk) \cdot \bsg)
		=
		\sum_{\bsg_{\{1:d-1\}} \in G_{b,m}^{d-1}} \sum_{g_d \in G_{b,m}} \delta_p(\tr_m(\bsk_{\{1:d-1\}}) \cdot \bsg_{\{1:d-1\}} + \tr_m(k_d) \,g_d)
		=
		(b^m)^{d-1}
		,
	\end{equation}
	which follows due to the fact that
	\begin{equation*}
		\sum_{g_d \in G_{b,m}} \delta_p(\tr_m(\bsk_{\{1:d-1\}}) \cdot \bsg_{\{1:d-1\}} + \tr_m(k_d) \, g_d) = 1
		,
	\end{equation*}
	since for $p$ irreducible $\F_b[x] / (p)$ is a finite field and thus 
	\begin{equation*}
		\tr_m(k_1) g_1 + \cdots + \tr_m(k_{d-1}) g_{d-1} + \tr_m(k_d) g_d 
		\equiv 
		0 \pmod{p}
	\end{equation*}
	has $g_d \equiv - \tr_m(k_d)^{-1} (\tr_m(k_1) g_1 + \cdots + \tr_m(k_{d-1}) g_{d-1}) \pmod{p}$ as the unique solution in $G_{b,m}$. 
	
	By the standard averaging argument, which implies that there is always at least one element in a set of real numbers which is as good as average, there exists a $\bar{\bsg} \in G_{b,m}^d$ which satisfies
	\begin{align} \label{eq:average}
		T_{\bsgamma}(\bar{\bsg},p)
		=
		\min_{\bsg \in G_{b,m}^d} T_{\bsgamma}(\bsg,p)
		&\le
		\frac{1}{(b^m)^d} \sum_{\bsg \in G_{b,m}^d} T_{\bsgamma}(\bsg,p) \nonumber \\
		&=
		\frac{1}{(b^m)^d} \sum_{\bszero \ne \bsk \in \{0,1,\ldots,b^m-1\}^d} \frac{1}{r_{1,\bsgamma}(\bsk)} \sum_{\bsg \in G_{b,m}^d} \delta_p(\tr_m(\bsk) \cdot \bsg) \nonumber \\
		&=
		\frac{1}{b^m} \sum_{\bszero \ne \bsk \in \{0,1,\ldots,b^m-1\}^d} \frac{1}{r_{1,\bsgamma}(\bsk)}
		,
	\end{align}
	where we used again the identity in \eqref{eq:ind_func}. Note that we can write
	\begin{align*}
		\sum_{\bszero \ne \bsk \in \{0,1,\ldots,b^m-1\}^d} \! \frac{1}{r_{1,\bsgamma}(\bsk)}
		&=
		\!\!\sum_{\emptyset\neq\setu\subseteq \{1:d\}} \sum_{\bsk_\setu \in \{1,\ldots,b^m-1\}^{\abs{\setu}}} \! \frac{1}{r_{1,\bsgamma}(\bsk_\setu)}
		=
		\!\!\sum_{\emptyset\neq\setu\subseteq \{1:d\}} \!\! \gamma_\setu \!\!\! \sum_{\bsk_\setu \in \{1,\ldots,b^m-1\}^{\abs{\setu}}} \prod_{j\in\setu} b^{-\psi_b(k_j)} \\	
		&=
		\!\!\sum_{\emptyset\neq\setu\subseteq \{1:d\}} \!\! \gamma_\setu \prod_{j\in\setu} \sum_{k_j \in \{1,\ldots,b^m-1\}} \!\!\!
		b^{-\psi_b(k_j)}
		= 
		\!\sum_{\emptyset\neq\setu\subseteq \{1:d\}} \gamma_\setu \left( \sum_{k=1}^{b^m-1} b^{-\psi_b(k)} \right)^{|\setu|} . 
	\end{align*}
	The inner sum of the latter term equals
	\begin{equation*}
		\sum_{k=1}^{b^m-1} b^{-\psi_b(k)}
		=
		\sum_{t=0}^{m-1} \sum\limits_{k=b^{t}}^{b^{t+1}-1} b^{-\lfloor \log_{b}k \rfloor} 
		=
		\sum_{t=0}^{m-1} (b-1) b^t \, b^{-t}
		=
		\sum_{t=0}^{m-1} (b-1)
		=
		m \, (b-1),
	\end{equation*} 
	such that we obtain
	\begin{equation*}
		\sum_{\bszero \ne \bsk \in \{0,1,\ldots,b^m-1\}^d} \frac{1}{r_{1,\bsgamma}(\bsk)}
		=
		\sum_{\emptyset\neq\setu\subseteq \{1:d\}} \gamma_\setu \left( m \, (b-1) \right)^{|\setu|}
		.
	\end{equation*}
	Combining this with \eqref{eq:average} yields the existence of a good generating vector in $G_{b,m}^d$ as claimed.
\end{proof}

The following proposition, which was proved in \cite{EKOS2020}, gives a bound on the difference between the worst-case error $e_{b^m,d,\alpha,\bsgamma}(\bsg)$ and the truncated quality measure $T_{\alpha,\bsgamma}(\bsg,p)$ of a polynomial lattice rule with 
generator $\bsg$ and modulus $p$.
 
\begin{proposition} \label{prop:trunc_error}
	Let $p \in \F_b[x]$ with $\deg(p)=m$, let $\bsgamma = (\gamma_{\setu})_{\setu \subseteq \{1:d\}}$ be positive weights, and let 
	$\bsg = (g_1,\ldots,g_d) \in G_{b,m}^d$ such that $\gcd(g_j,p)=1$ for all $j\in\{1,\ldots,d\}$. Then, for any $\alpha>1$ and for $N=b^m$, we have
	\begin{align*}
		e_{b^m,d,\alpha,\bsgamma}(\bsg) - T_{\alpha,\bsgamma}(\bsg,p)
		\le
		\frac{1}{N^{\alpha}} \sum_{\emptyset\neq \setu \subseteq \{1:d\}} \gamma_{\setu} (2 \mu_b(\alpha))^{\abs{\setu}}
		.
	\end{align*}
\end{proposition}

Combining the results of Theorem \ref{thm:exist_T} and Proposition \ref{prop:trunc_error}, we can immediately deduce the existence of good polynomial lattice rules with respect to the worst-case error in the space $W_{d,\bsgamma}^\alpha$. We state the corresponding result, which first appeared in \cite{EKOS2020}, in the theorem below. 

\begin{theorem} \label{thm:exist_wce}
	Let $p \in \F_b[x]$ be an irreducible polynomial with $\deg(p)=m$, let $N=b^{m}$, and let 
	$\bsgamma = (\gamma_{\setu})_{\setu \subseteq \{1:d\}}$ be a sequence 
	of positive weights. Then there exists a $\bsg \in G_{b,m}^d$ such that, for all $\alpha >1$, 
	the worst-case error $e_{b^m,d,\alpha,\bsgamma}(\bsg)$ satisfies
	\begin{align*}
	e_{b^m,d,\alpha,\bsgamma}(\bsg) 
	\le 
	\frac{1}{N^\alpha} \left(\sum_{\emptyset\neq \setu \subseteq \{1:d\}} \gamma_{\setu} (2\mu_b(\alpha))^{|\setu|} 
	+ \left(\sum_{\emptyset\neq \setu \subseteq \{1:d\}} \gamma_{\setu}^{1/\alpha}(m(b-1))^{|\setu|}\right)^{\alpha}\right)
	.
	\end{align*}
\end{theorem}

While this result is interesting from a theoretical perspective, the involved argument is non-constructive. In the remainder of this article, we will device explicit algorithms for the construction of good polynomial lattice rules. 
	
\section{The CBC construction method for polynomial lattice rules}\label{sec:CBC}

In this section, we formulate and analyze a component-by-component search algorithm for the construction of good polynomial lattice rules. 
The advantage of the presented algorithm is that the involved quality criterion is independent of the smoothness parameter $\alpha$.

\subsection{Preliminary results}

At first, we summarize a number of auxiliary statements which will be needed in the following analysis. The lemmas are 
taken from \cite{EKNO2020} and \cite{EKOS2020}, respectively, where construction algorithms of a similar 
nature as the proposed method have been studied. \\

Consider at first the following Walsh series for $x\in (0,1)$, based on the decay function $r_1$,
\begin{equation*}
	\sum_{k=0}^\infty \frac{\wal_k(x)}{r_1(k)}
	=
	1 + \sum_{k=1}^\infty \frac{e^{2 \pi \icomp (\kappa_{0,k} \xi_1 + \kappa_{1,k} \xi_2 + \cdots )/b}}{b^{\floor{\log_b(k)}}},
\end{equation*}
which, as we shall see, is closely related to the quality measure $T_{\bsgamma}$ in \eqref{eq:qual_meas}. Here and in the following lemma, we use similar base $b$ expansions as introduced above, namely $k = \kappa_{0,k} + \kappa_{1,k} b + \cdots+ \kappa_{a,k} b^{a}$ for $k\in \N_0$ and $x=\xi_1 b^{-1} + \xi_2 b^{-2} + \cdots$ for $x\in [0,1)$.

\begin{lemma} \label{lem:Walsh-series}
	For base $b \ge 2$, the Walsh series of $-(b-1) (\floor{\log_b(x)} + 1)$ equals, pointwise for $x\in (0,1)$,
	\begin{equation*}
	-(b-1) (\floor{\log_b(x)} + 1)
	=
	1 + \sum_{k=1}^\infty \frac{e^{2 \pi \icomp (\kappa_{0,k} \xi_1 + \kappa_{1,k} \xi_2 + \cdots )/b}}{b^{\floor{\log_b(k)}}} 
	=
	\sum_{k=0}^\infty \frac{ \wal_k(x)}{r_{1}(k)} .
	\end{equation*}
\end{lemma}

Based on the result in Lemma \ref{lem:Walsh-series}, it was shown in \cite{EKOS2020} that the function $-(b-1) (\floor{\log_b(x)} + 1)$ 
can be written in terms of its truncated Walsh series with uniformly bounded remainder term.

\begin{lemma}\label{lem:trunc_walsh_series}
	Let $N=b^m$ with $m\in \N$ and base $b \ge 2$. Then for any $x \in (0,1)$ there exists a $\tau(x)\in\R$ with 
	$|\tau(x)|< \frac{b}{b-1}$ such that
	\begin{equation*}
		-(b-1) (\floor{\log_b(x)} + 1)
		=
		\sum_{k=0}^{N-1} \frac{\wal_k(x)}{r_{1}(k)} + \frac{\tau(x)}{N x}
		.
	\end{equation*}
\end{lemma}

We will also make use of the following lemma, which was shown in \cite{EKNO2020}.

\begin{lemma} \label{lem:diff_prod}
	For $1\le j \le d$, let $u_j, v_j$, and $\rho_j$ be real numbers satisfying
	\begin{align*}
		(a) \quad u_j = v_j + \rho_j, \quad
		(b) \quad |u_j| \le \bar{u}_j, \quad
		(c) \quad \bar{u}_j \geq 1 ,
	\end{align*}
	for all $j\in\{1 \mcol d\}$. Then, for any set $\emptyset \ne \setu \subseteq \{1 \mcol d\}$, there exists a $\theta_{\setu}$ 
	with $|\theta_{\setu}| \le 1$ such that 
	\begin{align*} 
		\prod_{j \in \setu} u_j
		&=
		\prod_{j \in \setu} v_j + \theta_{\setu} \left(\prod_{j \in \setu} (\bar{u}_j+|\rho_j|) \right) \sum_{j \in \setu} |\rho_j|
		.
	\end{align*}
\end{lemma}

Furthermore, we recall the character property of Walsh functions for polynomial lattice rules with prime base $b$. Let $P(\bsg,p)=\{\bsx_0,\ldots,\bsx_{b^m-1}\} \subset [0,1]^d$ be a polynomial lattice with generating vector $\bsg \in (\F_b[x])^d$ and 
modulus $p \in \F_b[x]$ with $\deg(p)=m$. Then, for a vector $\bsk \in \N_0^d$, the following identity holds:
\begin{equation} \label{eq:char-prop}
	\frac{1}{b^m} \sum_{n=0}^{b^m-1} \wal_{\bsk} (\bsx_n)
	=
	\delta_p(\tr_m(\bsk) \cdot \bsg)
	=
	\begin{cases}
	1, & \text{if } \tr_m(\bsk) \cdot \bsg \equiv 0 \pmod{p} , \\
	0, & \text{otherwise} .
	\end{cases}
\end{equation}
We remark that an analogous result to \eqref{eq:char-prop} also holds if we only consider projections of the polynomial lattice and the generating vector onto a non-empty subset $\setu \subseteq \{1 \mcol d\}$, as also the projection of a polynomial lattice is a polynomial lattice that is generated by the corresponding projection of the generating vector.

Additionally, the following two results will be helpful for the subsequent analysis.
\begin{lemma}\label{lem:projections_PLR}
	Let $P(\bsg,p)\subset [0,1]^d$ be a polynomial lattice with modulus $p \in \F_b[x]$ with $\deg(p)=m$ and generating vector 
	$\bsg=(g_1,\ldots,g_d) \in (\F_b[x])^d$ such that $\gcd(g_j,p)=1$ for $1 \le j \le d$. Then each one-dimensional projection of 
	$P(\bsg,p)$ is the full grid 
	\begin{equation*}
		\left\{0,\frac{1}{b^m}, \ldots, \frac{b^m-1}{b^m}\right\} ,
	\end{equation*}
	and in particular the projection of the point with index $0$ is always $0$. 
\end{lemma}

\begin{proof}
	The result follows from Definition \ref{def:poly_lat} and \cite[Remark 10.3]{DP10}.
\end{proof}

\begin{lemma}\label{lem:sum_PLR}
	Let $P(\bsg,p)=\{\bsx_0,\ldots,\bsx_{b^m-1}\} \subset [0,1]^d$ be a polynomial lattice point set with modulus $p \in \F_b[x]$ with $\deg(p)=m$ 
	and generating vector $\bsg \in (\F_b[x])^d$ such that $\gcd(g_j,p)=1$ for $1 \le j \le d$. Furthermore, let $m\ge 4$. For each point $\bsx_n$ with 
	$n=0,1,\ldots,b^m-1$, we denote its coordinates by $\bsx_n = (x_{n,1},\ldots,x_{n,d})$. Then, for any $j \in \{1,\ldots,d\}$, it is true that
	\begin{equation*}
		\frac{1}{b^m} \sum_{n=1}^{b^m-1} \frac{1}{x_{n,j}}
		<
		1 + m \log b \le m (b-1)
		.
	\end{equation*}
\end{lemma}

\subsection{An alternative CBC construction for polynomial lattice rules}

In this section, we study a component-by-component (CBC) construction for polynomial lattice rules which is based on a quality criterion related to the quantity $T_{\bsgamma}$. In \cite{EKNO2020} such an algorithm was analyzed for (ordinary) lattice rules. Throughout this section, we will assume that the modulus $p \in \F_b[x]$ is irreducible. 

Concerning the weights, our algorithm can, as indicated in one of our main results (Corollary~\ref{cor:cbc}), be run with respect to the weights $\bsgamma^{1/\alpha} = (\gamma_{\setu}^{1/\alpha})_{\setu \subseteq \{1:d\}}$ to obtain a polynomial lattice rule that yields a low worst-case error in the Walsh space $W_{d,\bsgamma}^\alpha$, or, alternatively, with respect to the weights $\bsgamma$ to obtain good polynomial lattice rules in the space $W_{d,\bsgamma^{\alpha}}^\alpha$. In the latter case, the construction algorithm is independent of the smoothness parameter $\alpha$ and we obtain worst-case error bounds that hold simultaneously for all $\alpha>1$.

In order to avoid confusion, we will therefore denote the weights in this section by $\bseta$ instead of $\bsgamma$ and outline the algorithm based on $\bseta$. In Corollary \ref{cor:cbc}, we will then choose $\bseta$ equal to $\bsgamma^{1/\alpha}$ or $\bsgamma$, respectively. Before we formulate the algorithm, we prove the following theorem.

\begin{theorem} \label{thm:CBC_split}
	Let $b$ be prime, let $m,d \in \N$ with $m \geq 4$, let $p\in \F_b[x]$ be an irreducible polynomial with $\deg(p)=m$, and let 
	$\bseta = (\eta_{\setu})_{\setu \subseteq \{1:d\}}$ be positive weights. Furthermore, let $\bsg = (g_1,\ldots,g_d) \in G_{b,m}^d$ be the generating vector of the polynomial lattice $P(\bsg,p)=\{\bsx_0,\ldots,\bsx_{b^m-1}\} \subset [0,1]^d$. Then the following estimate holds:
	\begin{align*}
		T_{\bseta}(\bsg,p) 
		&\le 
		\sum_{\emptyset\ne \setu \subseteq \{1:d\}} \frac{\eta_\setu}{b^m} \sum_{n=1}^{b^m-1} L_\setu 
		\left(v_m\left(\frac{n(x) \,\bsg_\setu (x)}{p(x)} \right)\right) 
		+ \sum_{\emptyset\ne \setu \subseteq \{1:d\}} \frac{\eta_\setu}{b^m} ((b-1)m)^{|\setu|}
		\\ 
		&\quad+ 
		\sum_{\emptyset\ne \setu \subseteq \{1:d\}} \frac{\eta_\setu}{b^m} \, (b \,m \abs{\setu})
		\left((b-1)m + \frac{b}{b-1} \right)^{\abs{\setu}} ,
	\end{align*}
	where for a subset $\emptyset \ne \setu \subseteq \{1\mcol d\}$, we define the function $L_\setu: (0,1]^{\abs{\setu}} \to \R$ by
	\begin{equation*}
		L_{\setu}(\bsx_\setu) 
		:=    
		\prod_{j\in\setu}\left((1-b) \floor{\log_b(x_j)} - b \right) 
		= 
		\sum_{\bsk_\setu \in \N^{|\setu|}} \frac{\wal_{\bsk_\setu}(\bsx_{\setu})}{\prod_{j\in\setu}b^{\floor{\log_b(k_j)}}}
		.
	\end{equation*}
\end{theorem}

\begin{proof}
	Using the character property of Walsh functions in \eqref{eq:char-prop}, we can rewrite $T_{\bseta}(\bsg,p)$ with the help of the identity in 
	Lemma \ref{lem:trunc_walsh_series}. First, recall that we have, for $k \in \N_0$, 
	\begin{equation*}
		r_1(k) 
		=
		\left\{\begin{array}{cc}
			1, & {\text{if }} k=0 , \\ 
			b^{\lfloor \log_b(k)\rfloor}, & {\text{if }} k \ne 0 .
		\end{array}\right.
	\end{equation*}
	For a point $\bsx_n$, $0\le n\le b^m-1$, and a set $\emptyset\ne \setu \subseteq \{1 \mcol d\}$, we write $\bsx_{n,\setu}$ to denote the projection of $\bsx_n$ onto the components with indices in $\setu$. Then, we obtain that
	\begin{align} \label{eq:estimate-T}
		T_\bseta(\bsg,p) 
		&=
		\sum_{\bszero \ne \bsk \in \{0,1,\ldots,b^m-1\}^d} \frac{\delta_p(\tr_m(\bsk) \cdot \bsg)}{r_{1,\bseta}(\bsk)} \notag \\
		&=  
		\sum_{\emptyset\ne \setu \subseteq \{1:d\}} \frac{\eta_\setu}{b^m} \sum_{n=0}^{b^m-1}
		\left[ \sum_{\bsk_\setu \in \{1,\dots,b^m-1\}^{|\setu|}} 
		\frac{\wal_{\bsk_\setu}(\bsx_{n,\setu})}{\prod_{j\in\setu} b^{\floor{\log_b (k_j)}}} \right]
		\notag \\
		&= 
		\sum_{\emptyset\ne \setu \subseteq \{1:d\}} \frac{\eta_\setu}{b^m} 
		\left( \sum_{\bsk_\setu \in \{1,\dots,b^m-1 \}^{|\setu|}} \! \frac{1}{\prod_{j\in\setu} b^{\floor{\log_b(k_j)}}} 
		\!+\!\! \sum_{n=1}^{b^m-1} \prod_{j\in\setu} \left[ \sum_{k=1}^{b^m-1}\frac{\wal_k(x_{n,j})}{b^{\floor{\log_b(k)}}} \right] \right) 
		\notag \\
		&=
		\sum_{\emptyset\ne \setu \subseteq \{1:d\}} \frac{\eta_\setu}{b^m} 
		\left( \sum_{\bsk_\setu \in \{1,\dots,b^m-1 \}^{|\setu|}} \! \frac{1}{\prod_{j\in\setu} b^{\floor{\log_b(k_j)}}} 
		\!+\!\! \sum_{n=1}^{b^m-1} \left[ \prod_{j\in\setu} v_j(n) - \prod_{j\in\setu} u_j(n) + \prod_{j\in\setu} u_j(n) \right] \right)
		\notag \\
		&= 
		\!\sum_{\emptyset\ne \setu \subseteq \{1:d\}} \frac{\eta_\setu}{b^m} 
		\sum_{\bsk_\setu \in \{1,\dots,b^m-1 \}^{|\setu|}} \frac{1}{\prod_{j\in\setu} b^{\floor{\log_b(k_j)}}} 
		+ \sum_{\emptyset\ne \setu \subseteq \{1:d\}} \frac{\eta_\setu}{b^m} \sum_{n=1}^{b^m-1} \prod_{j\in\setu} u_j(n)
		\notag \\
		&\quad-
		\sum_{\emptyset\ne \setu \subseteq \{1:d\}} \frac{\eta_\setu}{b^m} 
		\sum_{n=1}^{b^m-1} \theta_{\setu}(n) 
		\left( \prod_{j\in\setu} (\bar{u}_j + |\rho_j(n)|) \right) \sum_{j\in\setu} |\rho_j(n)|
		\notag \\
		&= 
		\!\sum_{\emptyset\ne \setu \subseteq \{1:d\}} \!\frac{\eta_\setu}{b^m} \sum_{n=1}^{b^m-1} L_\setu 
		\left(v_m\left(\frac{n(x) \,\bsg_\setu(x)}{p(x)} \right)\right)
		+
		\sum_{\emptyset\ne \setu \subseteq \{1:d\}} \frac{\eta_\setu}{b^m} 
		\!\sum_{\bsk_\setu \in \{1,\dots,b^m-1 \}^{|\setu|}} \frac{1}{\prod_{j\in\setu} b^{\floor{\log_b(k_j)}}} 
		\notag \\
		&\quad-
		\sum_{\emptyset\ne \setu \subseteq \{1:d\}} \frac{\eta_\setu}{b^m} 
		\sum_{n=1}^{b^m-1} \theta_{\setu}(n) 
		\left( \prod_{j\in\setu} (\bar{u}_j + |\rho_j(n)|) \right) \sum_{j\in\setu} |\rho_j(n)|,
	\end{align}
	where we used that $x_{0,j}=0$ for all $1 \le j \le d$ and Lemma \ref{lem:diff_prod} with
	\begin{align*}
		u_j = u_j(n) &:= -(b-1) (\floor{\log_b(x_{n,j})} + 1) - 1 ,
		&
		\bar{u}_j=\bar{u}_j(n) &:= (b-1)m ,
		\\
		v_j = v_j(n) &:= \sum_{k=1}^{b^m-1} \frac{\wal_k(x_{n,j})}{b^{\floor{\log_b(k)}}} ,
		&
		\rho_j = \rho_j(n) &:= \frac{\tau_j(n)}{x_{n,j} \, b^m}
		,
	\end{align*}
	with $|\theta_\setu(n)| \le 1$ and $\tau_j (n)=\tau_j (x_{n,j})$ with $|\tau_j(n)| \le \frac{b}{b-1}$ for all $n\in\{0,1,\ldots,b^m-1\}$. 
	Due to Lemma \ref{lem:trunc_walsh_series}, Condition~(a) of Lemma \ref{lem:diff_prod} is fulfilled. Furthermore, we recall that by Lemma \ref{lem:projections_PLR} the one-dimensional projections of the points $\bsx_n$ with indices $n \ge 1$ satisfy
	\begin{equation*}
		x_{n,j} 
		=
		v_m\left(\frac{n(x) g_j(x)}{p(x)}\right)
		\ge
		v_m\left(\frac{1}{x^m}\right)
		=
		b^{-m}
	\end{equation*}
	for all $1 \le n \le b^m-1 $ and all $1 \le j \le d$ such that we obtain
	\begin{equation*}
		\abs{u_j(n)}
		\le
		-(b-1) (\floor{\log_b(b^{-m})} + 1)
		= 
		-(b-1)(-m+1) 
		<
		(b-1)m
		= 
		\bar{u}_j
		\quad 
	\end{equation*}
	with $\bar{u}_j \ge 1$ and so Conditions~(b) and~(c) of Lemma \ref{lem:diff_prod} are also fulfilled.

	As in the proof of Theorem \ref{thm:exist_T}, we can express the second sum in \eqref{eq:estimate-T} as
	\begin{equation*}
		\sum_{\emptyset\ne \setu \subseteq \{1:d\}} \frac{\eta_\setu}{b^m} 
		\sum_{\bsk_\setu \in \{1,\dots,b^m-1 \}^{|\setu|}} \frac{1}{\prod_{j\in\setu} b^{\floor{\log_b(k_j)}}}
		=
		\frac{1}{b^m} \sum_{\emptyset\ne \setu \subseteq \{1:d\}} \eta_\setu \,((b-1)m)^{|\setu|}
		,
	\end{equation*}
	while the absolute value of the third sum in \eqref{eq:estimate-T} can be bounded as follows,
	\begin{align*}
		&\abs{ \sum_{\emptyset\ne \setu \subseteq \{1:d\}} \frac{\eta_\setu}{b^m} 
		\sum_{n=1}^{b^m-1} \theta_{\setu}(n) 
		\left( \prod_{j\in\setu} (\bar{u}_j + |\rho_j(n)|) \right) \sum_{j\in\setu} |\rho_j(n)| }
 		\\
		&\quad=
		\abs{ \sum_{\emptyset\ne \setu \subseteq \{1:d\}} \frac{\eta_\setu}{b^m}
		\sum_{n=1}^{b^m-1} \theta_{\setu}(n) \left( \prod_{j\in\setu} \left((b-1)m + \frac{|\tau_j(n)|}{x_{n,j} \, b^m}\right) \right) 
		\sum_{j\in\setu} \frac{|\tau_j(n)|}{x_{n,j} \, b^m} }
		\notag \\
		&\quad\le
		\sum_{\emptyset\ne \setu \subseteq \{1:d\}} \frac{\eta_\setu}{b^m}
		\sum_{n=1}^{b^m-1} |\theta_{\setu}(n) |\left( \prod_{j\in\setu} \left((b-1)m + \frac{b}{b-1} \right) \right) 
		\sum_{j\in\setu} \frac{|\tau_j(n)|}{x_{n,j} \, b^m}
		\notag \\
		&\quad\le
		\sum_{\emptyset\ne \setu \subseteq \{1:d\}} \frac{\eta_\setu}{b^m}
		\left( \prod_{j\in\setu} \left((b-1)m + \frac{b}{b-1} \right) \right) 
		\sum_{j\in\setu} \sum_{n=1}^{b^m-1} \frac{b}{(b-1) b^m} \frac{1}{x_{n,j}}
		\notag \\
		&\quad\le	
		\sum_{\emptyset\ne \setu \subseteq \{1:d\}} \frac{\eta_\setu}{b^m}
		\left( \prod_{j\in\setu} \left((b-1)m + \frac{b}{b-1} \right) \right) \frac{b}{b-1}
		\sum_{j\in\setu} m (b-1)
		\notag \\
		&\quad=
		\frac{1}{b^m} \sum_{\emptyset\ne \setu \subseteq \{1:d\}} \eta_\setu \, (b \,m \abs{\setu})
		\left((b-1)m + \frac{b}{b-1} \right)^{\abs{\setu}}
		,
	\end{align*}
	where we used Lemma \ref{lem:sum_PLR} and the fact that $x_{n,j} \ge b^{-m}$ for each $j$ and all $1 \le n < b^m$.
	Then combining the estimates yields the claimed result.
\end{proof}

Theorem \ref{thm:CBC_split} indicates that it is reasonable to search for generating vectors $\bsg$ such that 
\begin{align*}
	\sum_{\emptyset\neq\setu\subseteq\{1:d\}} 
	\frac{\eta_\setu}{b^m}\sum_{n=1}^{b^m-1} L_\setu \left(v_m\left(\frac{n(x) \, \bsg_\setu(x)}{p(x)}\right)\right)
\end{align*}
is small, as then also $T_{\bseta}(\bsg_,p)$ is sufficiently small. We will now give the definition of the quality function 
that we try to minimize in the component-by-component algorithm.

\begin{definition}[Quality Function] \label{eq:qual_fun}
	For a generating vector $\bsg = (g_1,\dots,g_d) \in (G^\ast_{b,m})^d$ with prime $b \ge 2$, positive integers $m,d \in \N$, positive weights 
	$\bseta = (\eta_\setu)_{\setu\subseteq\{1:d\}}$, and an  irreducible polynomial $p \in \F_b[x]$ with $\deg(p) = m$, we define the quality function
	\begin{equation*}
		K_{b^m,d,\bseta}(\bsg) =K_{b^m,d,\bseta}(\bsg, p)
		:=
		\sum_{\emptyset\neq\setu \subseteq \{1:d\}} \eta_{\setu} \sum_{n=1}^{b^m-1} L_{\setu} 
		\left(v_m\left(\frac{n(x) \, \bsg_\setu(x)}{p(x)}\right)\right)
		.
	\end{equation*}
\end{definition}
\noindent
Based on this quality function, we formulate the following component-by-component construction.

\begin{algorithm}[H] 
	\caption{Component-by-component construction}	
	\label{alg:cbc}
	\vspace{5pt}
	\textbf{Input:} Prime number $b$, integers $m,d \in \N$, and positive weights $\bseta=(\eta_\setu)_{\setu\subseteq \{1:d\}}$. \\
	\vspace{-10pt}
	\begin{algorithmic}
		\STATE Choose an irreducible polynomial $p \in \F_b[x]$ with $\deg(p) = m$.
		\vspace{7pt}
		\STATE Set $g_1 \equiv 1$.
		\FOR{$s=2$ \bf{to} $d$}
		\STATE $g_s = \argmin\limits_{g \in G_{b,m}^*} K_{b^m,s,\bseta}(g_1,\ldots,g_{s-1},g)$
		\ENDFOR
	\end{algorithmic}
	\vspace{5pt}
	\textbf{Return:} Generating vector $\bsg=(g_1,\ldots,g_d) \in (G_{b,m}^\ast)^d$.
\end{algorithm}
In the following section, we analyze the worst-case error behavior of polynomial lattice rules with generating vectors constructed by Algorithm~\ref{alg:cbc}.

\subsection{Error bounds for the obtained polynomial lattice rules}

In the next theorem, we show an upper bound on the quantity $K_{b^m,d,\bseta}(\bsg)$ for $\bsg$ constructed by Algorithm \ref{alg:cbc}.

\begin{theorem} \label{thm:cbc_bound}
	Let $b$ be prime, let $m,d \in \N$, let $p \in \F_b[x]$ be irreducible with $\deg(p) = m$, and let $\bseta = (\eta_\setu)_{\setu\subseteq\{1:d\}}$ be positive weights. Then the corresponding generating vector $\bsg$ constructed by Algorithm~\ref{alg:cbc} satisfies
	\begin{equation*}
		K_{b^m,s,\bseta}(\bsg) 
		\le 
		\sum_{\emptyset\ne \setu \subseteq \{1:s\}} \eta_\setu ((b-1) m)^{|\setu|}
	\end{equation*}
	for every $s\in\{1,\ldots,d\}$.
\end{theorem}

\begin{proof}
	We prove the result by induction on $s \in \{1,\dots,d\}$. Using that $g_1 \equiv 1$, we obtain for $s=1$ that
	\begin{align*}
		K_{b^m,1,\bseta} (g_1) 
		= 
		K_{b^m,1,\bseta}(1) 
		&= 
		\eta_{\{1\}} \sum_{n=1}^{b^m-1} L_{\{1\}} \left(v_m\left(\frac{n(x)}{p(x)}\right)\right)
		\\
		&= 
		\eta_{\{1\}} \sum_{n=1}^{b^m-1}\left( -(b-1)\floor{\log_b\left(v_m\left(\frac{n(x)}{p(x)}\right)\right)} - b\right)
		\\
		&= 
		- \eta_{\{1\}} (b-1) \sum_{n=1}^{b^m-1} \floor{\log_b\left(v_m\left(\frac{n(x)}{p(x)}\right)\right)} - 
		\eta_{\{1\}} \, b \,(b^m-1)
		.
	\end{align*}
	In order to deal with the sum in the latter expression, we first observe that for any fraction 
	$\frac{n}{b^m}$ with $b^t \le n < b^{t+1}$ for $0 \le t < m$ we have that	
	\begin{equation*}
		\log_b \left(\frac{b^t}{b^m}\right) 
		\le 
		\log_b \left(\frac{n}{b^m}\right)
		<
		\log_b \left(\frac{b^{t+1}}{b^m}\right)
		.
	\end{equation*}
	Therefore, we find that
	\begin{equation*}
		-(m-t) \le \log_b \left(\frac{n}{b^m}\right) < -(m-t-1),
	\end{equation*}
	and hence $\lfloor \log_b (n/b^m) \rfloor= - (m-t)$. This identity and the application of Lemma \ref{lem:projections_PLR} for the case $d=1$ yield
	\begin{equation} \label{eq:estimate_log_b}
		-\sum_{n=1}^{b^m-1} \floor{\log_b\left(v_m\left(\frac{n(x)}{p(x)}\right)\right)}
		= 
		-\sum_{n=1}^{b^m-1} \floor{\log_b\left(\frac{n}{b^m}\right)}
		= 
		\sum_{t=0}^{m-1} (b-1) \, b^t \, (m-t) 
		= 
		\frac{b^{m+1}  - b m + m - b}{b-1}
	\end{equation}
	since there are exactly $(b-1) b^t$ numbers $n$ with $b^t \le n < b^{t+1}$. Therefore, we obtain that
	\begin{equation*}
		K_{b^m,1,\bseta} (g_1) 
		= 
		\eta_{\{1\}} (b^{m+1} - bm + m - b - b^{m+1} + b)
		= 
		-\eta_{\{1\}} (b-1) m
		<
		\eta_{\{1\}} (b-1) m
		.
	\end{equation*}
	
	Consider then $s\ge 2$ and assume that the claimed statement holds for $s-1$, that is, 
	\begin{equation*}
		K_{b^m,s-1,\bseta}(g_1,\dots,g_{s-1})
		\le 
		\sum_{\emptyset\ne \setu \subseteq \{1:s-1\}} \eta_\setu((b-1)m)^{|\setu|} .
	\end{equation*}
	By the standard averaging argument and the induction assumption we obtain
	\begin{align*}
		&K_{b^m,s,\bseta}(g_1,\dots,g_{s-1},g_s) 
		\le \frac{1}{b^m-1} \sum_{g\in G_{b,m}^\ast} K_{b^m,s,\bseta}(g_1,\dots,g_{s-1},g) 
		\\
		&\quad=
		\frac{1}{b^m-1} \sum_{g\in G^\ast_{b,m}} \left[ \sum_{\emptyset\ne \setu \subseteq \{1:s-1\}} \eta_\setu \sum_{n=1}^{b^m-1} 
		L_\setu \left(v_m\left(\frac{n(x)\, \bsg_\setu(x)}{p(x)} \right) \right) \right.
		\\
		&\qquad\qquad\qquad\qquad\quad+ \left. \sum_{\setv \subseteq \{1:s-1\}} \eta_{\setv\cup\{s\}}\sum_{n=1}^{b^m-1} 
		L_{\setv\cup\{s\}} 
		\left(v_m \left(\frac{n(x) \, \bsg_{\setv}(x)}{p(x)}\right), v_m \left(\frac{n(x) \, g(x)}{p(x)}\right)\right) \right] 
		\\
		&\quad= 
		K_{b^m,s-1,\bseta}(\bsg_{\{1:s-1\}}) + \frac{1}{b^m-1} 
		\sum_{\setv \subseteq \{1:s-1\}} \eta_{\setv\cup\{s\}} 
		\sum_{n=1}^{b^m-1} L_{\setv}\left( v_m\left(\frac{n(x) \,\bsg_\setv(x)}{p(x)} \right) \right) \times
		\\ 
		&\quad\qquad\qquad\qquad\qquad\qquad\,\,\,\, \times 
		\sum_{g\in G^*_{b,m}} \left( -(b-1) \floor{\log_b\left(v_m\left(\frac{n(x) \, g(x) }{p(x)} \right) \right)} - b \right).
	\end{align*}
	Using an argument similar to that which lead to \eqref{eq:estimate_log_b}, we obtain 
	\begin{equation*}
		 \sum_{g\in G^*_{b,m}} \!\!\! \left( -(b-1) \floor{\log_b\left(v_m\left(\frac{n(x) \, g(x) }{p(x)} \right) \right)} - b \right)
		 = 
		 b^{m+1}- m (b-1) -b -b (b^m-1) = -m(b-1)
		 ,
	\end{equation*}
	which in turn gives
	\begin{align*}
		&K_{b^m,s,\bseta}(g_1,\dots,g_{s-1},g_s) 
		\\
		&\le
		K_{b^m,s-1,\bseta}(\bsg_{\{1:s-1\}}) + \frac{1}{b^m-1} \sum_{\setv \subseteq \{1:s-1\}} \eta_{\setv\cup\{s\}} \sum_{n=1}^{b^m-1} L_{\setv}\left( v_m\left(\frac{n(x) \, \bsg_\setv(x)}{p(x)} \right) \right)(-m(b-1))
		\\
		&\le 
		\sum_{\emptyset\ne \setu \subseteq \{1:s-1\}} \eta_\setu((b-1)m)^{|\setu|} 
		+ \frac{m(b-1)}{b^m-1} \sum_{n=1}^{b^m-1} \sum_{\setv\subseteq\{1:s-1\}} \eta_{\setv\cup\{s\}} 
		\left|L_\setv \left( v_m\left( \frac{n(x)\,\bsg_\setv(x)}{p(x)} \right) \right)\right|
		.
	\end{align*}
	 Considering the term
	\begin{equation*}
		\left|L_\setv \left( v_m\left( \frac{n(x)\,\bsg_\setv(x)}{p(x)} \right) \right)\right|
		= 
		\prod_{j\in\setv} \left| -(b-1) \floor{\log_b\left( v_m\left(\frac{n(x) \, g_j(x)}{p(x)} \right) \right)} - b \right|
		,
	\end{equation*}
	we see that since 
	$1-b^{-m} \ge v_m\left(\frac{n(x) \, g(x)}{p(x)} \right) \ge b^{-m}$ for $g \in G_{b,m}^\ast$ 
	and $\deg(n) < m$ we have
	\begin{equation*}
		b-1
		\le
		-(b-1) \floor{\log_b\left( v_m\left(\frac{n(x) \,g(x)}{p(x)} \right) \right) } 
		\le 
		(b-1) m
	\end{equation*}
	and thus
	\begin{equation*}
		\left|L_\setv \left( v_m\left( \frac{n(x)\,\bsg_\setv(x)}{p(x)} \right) \right)\right|
		\le
		(m(b-1))^{|\setv|}
		.
	\end{equation*}
	Hence, the estimation yields
	\begin{align*}
		K_{b^m,s,\bseta}(g_1,\dots,g_{s-1},g_s) 
		&\le\!\!
		\sum_{\emptyset\ne \setu \subseteq \{1:s-1\}} \!\!\! \eta_\setu((b-1)m)^{|\setu|} 
		+ \frac{m(b-1)}{b^m-1} \sum_{n=1}^{b^m-1} \! \sum_{\setv \subseteq \{1:s-1\}} \!\!\! \eta_{\setv\cup\{s\}} (m(b-1))^{|\setv|}
		\\
		&=\!\!
		\sum_{\emptyset\ne \setu \subseteq \{1:s-1\}} \!\!\! \eta_\setu((b-1)m)^{|\setu|} 
		+ \sum_{\setv \subseteq \{1:s-1\}} \!\!\! \eta_{\setv\cup\{s\}} (m(b-1))^{|\setv|+1}
		\\
		&=\!\! 
		\sum_{\emptyset\neq \setu \subseteq \{1:s\}} \eta_\setu((b-1)m)^{|\setu|} 
		,
	\end{align*}
	which is the claimed result.
\end{proof}

Theorem \ref{thm:cbc_bound} allows us to prove the main result regarding the construction in Algorithm \ref{alg:cbc}.

\begin{theorem} \label{thm:cbc-T}
	Let $b$ be prime, let $m,d \in \N$ with $m \geq 4$, let $p \in \F_b[x]$ be irreducible with $\deg(p) = m$, and let
	$\bseta = (\eta_\setu)_{\setu\subseteq\{1:d\}}$ be positive weights. Furthermore, let $\bsg \in (G^\ast_{b,m})^d$ be 
	the generating vector constructed by Algorithm~\ref{alg:cbc}. Then the following estimate holds:
	\begin{equation*}
		T_{\bseta}(\bsg,p)
		\le
		\frac{1}{b^m} \sum_{\emptyset\neq \setu \subseteq \{1:d\}} \eta_\setu \, 
		\left( 2 ((b-1)m)^{|\setu|} + b \,m \left(2 (b-1)m + \frac{2b}{b-1} \right)^{\abs{\setu}} \right)
		.
	\end{equation*}
	Moreover, if the weights $\bseta$ satisfy
	\begin{equation*}
		\sum_{j \ge 1} \max_{\setv \subseteq \{1 \mcol j-1\}} \frac{\eta_{\setv \cup \{j\}}}{\eta_\setv}
		<
		\infty
		,
	\end{equation*}
	then $T_{\bseta}(\bsg,p)$ can be bounded independently of the dimension $d$.
\end{theorem}

\begin{proof}
	Using the bound on $T_{\bseta}(\bsg,p)$ in Theorem \ref{thm:CBC_split} and inserting for $\bsg$ the generating vector
	obtained from Algorithm \ref{alg:cbc}, for which the bound on $K_{b^m,d,\bseta}(\bsg)$ in Theorem \ref{thm:cbc_bound} 
	holds, yields
	\begin{align*}
		T_{\bseta}(\bsg,p) 
		&\le 2\sum_{\emptyset\neq \setu \subseteq \{1:d\}} \frac{\eta_\setu}{b^m} ((b-1)m)^{|\setu|}
		+ 
		\sum_{\emptyset\ne \setu \subseteq \{1:d\}} \frac{\eta_\setu}{b^m} \, (b \,m \abs{\setu})
		\left((b-1)m + \frac{b}{b-1} \right)^{\abs{\setu}}
		\\
		&\le
		\frac{1}{b^m} \sum_{\emptyset\neq \setu \subseteq \{1:d\}} \eta_\setu \, 
		\left( 2 ((b-1)m)^{|\setu|} + b \,m \left(2 (b-1)m + \frac{2b}{b-1} \right)^{\abs{\setu}} \right)
		,
	\end{align*}
	where we used that $\abs{\setu} \le 2^{\abs{\setu}}$. This proves the first claim. Furthermore, we easily find that
	\begin{align*}
		b^m \,T_{\bseta}(\bsg,p) 
		&\le
		\sum_{\emptyset\neq \setu \subseteq \{1:d\}} \eta_\setu
		\left( 2 ((b-1)m)^{|\setu|} + b m \left(2 (b-1)m + \frac{2b}{b-1} \right)^{\abs{\setu}} \right)
		\\
		&\le
		 \sum_{\emptyset\neq \setu \subseteq \{1:d\}} \eta_\setu (4bm)^{|\setu|} 
		+ b m \sum_{\emptyset\neq \setu \subseteq \{1:d\}} \eta_\setu \left(4 b m \right)^{\abs{\setu}}
		\\
		&=
		(1+ b m) \sum_{\emptyset\neq \setu \subseteq \{1:d\}} \eta_\setu \left(4 b m \right)^{\abs{\setu}}
		\le
		C(\delta/2) \, b^{m \delta / 2} \sum_{\emptyset\neq \setu \subseteq \{1:d\}} \eta_\setu \left(4 b m \right)^{\abs{\setu}}
	\end{align*}
	for arbitrary $\delta>0$, where $C(\delta/2)$ is a constant depending only on $\delta$. We can now directly use the result in 
	\cite[Lemma 4]{EKNO2020} with $N=b^m$ and $a = 4 b / \log b$ to see that the sum in the last expression is of order 
	$\calO(b^{m \delta / 2})$. This yields the claimed result.
\end{proof}

Theorem \ref{thm:cbc-T} immediately implies the following result with respect to the worst-case error.

\begin{corollary} \label{cor:cbc}
	Let $b$ be prime, let $m,d \in \N$ with $m\ge 4$ and $N=b^m$, let $p \in \F_b[x]$ be irreducible with $\deg(p) = m$, and let 
	$\bsgamma = (\gamma_\setu)_{\setu\subseteq\{1:d\}}$ be positive weights, satisfying
	\begin{equation*}
		\sum_{j \ge 1} \max_{\setv \subseteq \{1 \mcol j-1\}} \frac{\gamma_{\setv \cup \{j\}}}{\gamma_\setv}
		<
		\infty
		.
	\end{equation*} 
	Then, for any $\delta>0$ and each $\alpha>1$, Algorithm \ref{alg:cbc}, run for the weights 
	$\bseta=\bsgamma = (\gamma_\setu)_{\setu\subseteq\{1:d\}}$, constructs a generating vector $\bsg \in (G^\ast_{b,m})^d$
	such that the worst-case error $e_{b^m,d,\alpha,\bsgamma^\alpha}(\bsg)$  satisfies
	\begin{equation*}
		e_{b^m,d,\alpha,\bsgamma^\alpha}(\bsg) 
		\le 
		\frac{1}{N^\alpha} \left(\sum_{\emptyset\neq \setu \subseteq \{1:d\}} 
		\gamma_{\setu}^\alpha(2\mu_b(\alpha))^{|\setu|} + \left(C_1(\bsgamma,\delta)\right)^{\alpha} \, N^{\alpha \delta} \right)
		, 
	\end{equation*}
	with a positive constant $C_1(\bsgamma,\delta)$ independent of $d$ and $N$.
	
	Additionally, if 
	\begin{equation*}
		\sum_{j \ge 1} \max_{\setv \subseteq \{1 \mcol j-1\}} \left( \frac{\gamma_{\setv \cup \{j\}}}{\gamma_\setv} \right)^{1/\alpha}
		<
		\infty
		,
	\end{equation*} 
	then Algorithm \ref{alg:cbc}, run for the weights $\bseta=\bsgamma^{1/\alpha} = (\gamma_\setu^{1/\alpha})_{\setu\subseteq\{1:d\}}$, 
	constructs a generating vector $\widetilde{\bsg} \in (G^\ast_{b,m})^d$ such that the worst-case error $e_{b^m,d,\alpha,\bsgamma}(\bsg)$ satisfies
	\begin{equation*}
		e_{b^m,d,\alpha,\bsgamma}(\widetilde{\bsg}) 
		\le 
		\frac{1}{N^\alpha} \left(\sum_{\emptyset\neq \setu \subseteq \{1:d\}} 
		\gamma_{\setu}(2\mu_b(\alpha))^{|\setu|} + \left(C_2(\bsgamma^{1/\alpha},\delta)\right)^{\alpha} \, N^{\alpha \delta} \right)
		, 
	\end{equation*}
	with a positive constant $C_2(\bsgamma^{1/\alpha},\delta)$ independent of $d$ and $N$.
\end{corollary}

\begin{proof}
	We know by Proposition \ref{prop:trunc_error} that the worst-case error satisfies 
	\begin{align} \label{eq:est_1_cor}
		e_{b^m,d,\alpha,\bseta^\alpha}(\bsg) 
		\le
		\frac{1}{N^{\alpha}} \sum_{\emptyset\neq \setu \subseteq \{1:d\}} 
		\eta_{\setu}^\alpha (2 \mu_b(\alpha))^{\abs{\setu}} + T_{\alpha,\bseta^\alpha}(\bsg,p).
	\end{align} 
	Then, using the fact that for $\alpha \ge 1$ we have $\sum_i x_i^\alpha \le \left( \sum_i x_i \right)^\alpha$ for $x_i \ge 0$ (which is sometimes referred to as Jensen's inequality) yields the estimate
	\begin{align} \label{eq:est_2_cor}
		T_{\alpha,\bseta^\alpha}(\bsg,p) \notag
		&= 
		\sum_{\mathbf{0}\neq\bsk\in A_p(\bsg)} (r_{\alpha,\bseta^\alpha}(\bsk))^{-1} 
		= 
		\sum_{\mathbf{0}\neq\bsk\in A_p(\bsg)} (r_{1,\bseta}(\bsk))^{-\alpha}
		\\ 
		&\le
		\left( \sum_{\mathbf{0}\neq\bsk\in A_p(\bsg)} (r_{1,\bseta}(\bsk))^{-1}\right)^{\alpha} = \left(T_{\bseta}(\bsg,p)\right)^{\alpha}
		.
	\end{align}
	By using the estimate obtained in Theorem \ref{thm:cbc-T}, we can deduce from this that, for arbitrary $\delta >0$,
	\begin{equation} \label{eq:est_3_cor}
		b^m T_{\bseta}(\bsg,p)
		\le
		C_1(\bseta,\delta) \, N^\delta,
	\end{equation}
	where $C_1(\bseta,\delta)$ is a constant depending only on $\delta$ and $\bseta$. Combining the obtained estimates in \eqref{eq:est_2_cor} and \eqref{eq:est_3_cor} with \eqref{eq:est_1_cor}, yields the claimed inequality for the choice $\bseta=\bsgamma$. The result for the choice $\bseta=\bsgamma^{1/\alpha}$ follows analogously.
\end{proof}

The result in Corollary \ref{cor:cbc} consists of two statements regarding the worst-case error behavior of generating vectors constructed by Algorithm \ref{alg:cbc}. On the one hand, when run with weights $\bsgamma^{1/\alpha}$, and hence depending on the parameter $\alpha$, the algorithm yields typical error bounds for the worst-case error in the space $W_{d,\bsgamma}^\alpha$. We emphasize that this type of result could also be obtained by formulating and using the CBC algorithm common in the literature (see below), which is instead directly based on the search criterion $e_{b^m,d,\alpha,\bsgamma}$. On the other hand, when run with weights $\bsgamma$, and thus independently of $\alpha$, the algorithm produces generating vectors for which bounds on the worst-case errors in the spaces $W_{d,\bsgamma^\alpha}^\alpha$ hold simultaneously for all $\alpha > 1$.

\begin{remark}
	In the recent article \cite{DG21}, it is shown that polynomial lattice rules which were constructed for the weighted Walsh space $W_{d,\bsgamma}^\alpha$ can also achieve the almost optimal convergence rate for the space $W_{d,\bsgamma'}^{\alpha'}$ with different smoothness parameter $\alpha' > 1$ and weight sequence $\bsgamma'$, provided that certain conditions on both weight sequences are satisfied. While in \cite{DG21} the relation between the different weight sequences and smoothness parameters may allow to transfer certain results, our algorithm (when run with weights $\bsgamma$) is independent of $\alpha$ and delivers QMC rules for which error bounds hold simultaneously for all $\alpha > 1$. Nevertheless, it would be interesting to investigate how the condition on the weight sequence $\bsgamma$ obtained here (see Corollary \ref{cor:cbc}) compares to the conditions in \cite{DG21}. We leave this question open for future research. 
\end{remark}

\section{Fast implementation of the CBC construction} \label{sec:fast_impl}

The fast component-by-component construction was first introduced in \cite{NC06b} for the case of (ordinary) lattice rules with a prime number of points and with the corresponding worst-case error as the quality criterion. A corresponding fast implementation for polynomial lattice rules was first analyzed in \cite{NC06a}. In this section, we discuss the efficient implementation of Algorithm \ref{alg:cbc} and analyze its complexity. Throughout the next two sections, we will consider the implementation of the CBC construction for the special case of product weights $\gamma_{\setu} = \prod_{j\in\setu} \gamma_j$ for a sequence of positive reals $(\gamma_j)_{j\ge1}$.

\subsection{Implementation and cost analysis of the CBC algorithm}

By using the same technique as in \cite{EKNO2020}, it is possible to rewrite the quality function $K_{b^m,d,\bsgamma}$ in Definition \ref{eq:qual_fun} for product weights $\bsgamma_\setu = \prod_{j\in \setu} \gamma_j$ with $(\gamma_j)_{j\geq 1} \subset \R_+$. For this purpose, 
let $N = b^m$ and consider $\bsg \in G_{b,m}^d$ with prime $b$ and positive integer $m$. Then, we see that $K_{N,d,\bsgamma}$ equals
\begin{align*}
	K_{N,d,\bsgamma}(\bsg)
	&=
	\sum_{\emptyset\neq\setu\subseteq\{1 \mcol d\}}\bsgamma_\setu \sum_{n=1}^{N-1} 
	L_\setu \left(v_m\left(\frac{n(x) \, \bsg_\setu(x)}{p(x)}\right)\right)
	\\
	&=
	\sum_{n=1}^{N-1} \left[ -1 + \prod_{j=1}^{d} \left( 1 + \gamma_j L_{\{j\}} 
	\left( v_m\left( \frac{n(x) \, g_j(x)}{p(x)} \right)\right) \right) \right] \\
	&=
	-(N-1) + \sum_{n=1}^{N-1} \prod_{j=1}^{d} \left( 1 + \gamma_j \left( (1-b)\left\lfloor\log_b\left(v_m\left(\frac{n(x) \,g_j(x)}{p(x)} \right) 
	\right) \right\rfloor - b \right) \right)
	.
\end{align*}

For short, we define the quantity
\begin{equation*}
	\bar{K}_{N,d,\bsgamma}(\bsg)
	:=
	\sum_{n=1}^{N-1} \prod_{j=1}^d 
	\left( 1 + \gamma_j \left( (1-b)\left\lfloor\log_b\left(v_m\left(\frac{n(x) \,g_j(x)}{p(x)} \right) \right) \right\rfloor - b \right) \right)
	.
\end{equation*}

We observe that as the term $N-1$ is constant, we can equivalently minimize the function $\bar{K}_{N,s,\bsgamma}(\bsg)$ instead of $K_{N,s,\bsgamma}(\bsg)$ in each step of Algorithm \ref{alg:cbc}. Noting that this function has the same structure as the worst-case error expression which is minimized in the standard CBC algorithm, see, e.g., \cite{NC06b}, we can employ the same machinery to obtain a fast implementation of Algorithm \ref{alg:cbc}. We summarize the computational cost of this fast implementation in the following proposition.

\begin{proposition} \label{prop:cost-CBC-alg}
	Let $m,d \in \N$ and set $N = b^m$ with prime $b$. For a given sequence of positive weights $\bsgamma = (\gamma_j)_{j=1}^d$, a generating vector $\bsg = (g_1,\dots,g_d)$ can be computed via Algorithm \ref{alg:cbc} using $\mathcal{O}(d \, N \log N)$ operations.
\end{proposition}

Since the fast implementation of Algorithm \ref{alg:cbc} can be done entirely analogously as for the standard CBC construction, 
we omit further implementation details and refer the reader to \cite{DKS13,NC06a}.

\section{Numerical results} \label{sec:num}

In this section, we present the results of some numerical experiments. Firstly, we compare the worst-case errors of generating vectors
constructed by the CBC algorithm with smoothness-independent quality function $K_{b^m,d,\bsgamma}$, i.e., Algorithm \ref{alg:cbc}, 
and the standard CBC algorithm for polynomial lattice rules (with the worst-case error as the quality function, as given in, e.g., \cite{N14,NC06a}) for several choices of positive weight sequences. Finally, we compare the computational costs of implementations 
of Algorithm \ref{alg:cbc} and the standard CBC algorithm for polynomial lattice rules.

The algorithms considered were all implemented using Python 3.6.3. The implementations are available in double-precision as well as in arbitrary-precision floating-point arithmetic with the latter provided by the multiprecision Python library mpmath.

\subsection{Error convergence behavior}

We considered the convergence rate of the worst-case error corresponding to a generating vector $\bsg$ constructed by Algorithm \ref{alg:cbc}  
and compared it to the error rates for polynomial lattice rules constructed by the standard CBC algorithm for different positive weight sequences $\bsgamma = (\gamma_j)_{j\ge 1}$. In particular, we computed the error $e_{b^m,d,\alpha,\bsgamma^\alpha}(\bsg)$ for dimension $d = 100$ for different values of $m$ and different values of the smoothness parameter $\alpha$. We would like to illustrate the almost optimal error rates of $\calO(N^{-\alpha+\delta})$ which can be achieved according to Corollary \ref{cor:cbc}, but may not be visible for the considered weights and the range of $m$ in our numerical experiments. Therefore, the presented graphs are to be understood as an illustration of the pre-asymptotic error behavior.

\begin{figure}[H]\small
	\centering
	\textbf{Error convergence in the space $W_{d,\bsgamma}^{\alpha}$ with $d=100, \alpha=1.5,2,3$.} \par\medskip 
	\hspace{-0.25cm}
	\centering
	\begin{subfigure}[b]{0.5\textwidth}
		\centering
		\begin{tikzpicture}

\begin{axis}[%
width=0.8\textwidth,
height=0.8\textwidth,
at={(0\textwidth,0\textwidth)},
scale only axis,
xmode=log,
xmin=42.6666666666667,
xmax=98304,
xminorticks=true,
xlabel style={font=\color{white!15!black}},
xlabel={Number of points $N=2^m$},
ymode=log,
ymin=1e-14,
ymax=1,
yminorticks=true,
ylabel style={font=\color{white!15!black}},
ylabel={Worst-case error $e_{N,d,\alpha,\mathbf{\gamma^{\alpha}}}(\mathbf{g})$},
axis background/.style={fill=white},
axis x line*=bottom,
axis y line*=left,
xmajorgrids,
xminorgrids,
ymajorgrids,
yminorgrids,
minor grid style={opacity=0},
legend style={at={(0.03,0.03)}, anchor=south west, legend cell align=left, align=left, draw=white!15!black}
]
\addplot [color=mycolor1, line width=0.9pt, mark=o, mark options={solid, mycolor1}, forget plot]
  table[row sep=crcr]{%
64	0.037624973273776\\
128	0.0152991284527486\\
256	0.00615396438059079\\
512	0.00247519771456294\\
1024	0.000991765322570814\\
2048	0.000394917502187285\\
4096	0.000157713889992297\\
8192	6.29653231131374e-05\\
16384	2.49659476038386e-05\\
32768	9.88785761504665e-06\\
65536	3.8615654141101e-06\\
};
\addplot [color=mycolor2, line width=0.9pt, mark=triangle, mark options={solid, mycolor2}, forget plot]
  table[row sep=crcr]{%
64	0.0376151186365635\\
128	0.0152933248325304\\
256	0.00618948770005742\\
512	0.00247447287359493\\
1024	0.0009979148524534\\
2048	0.000395689821292074\\
4096	0.00015735848167537\\
8192	6.27210633415442e-05\\
16384	2.49346925597582e-05\\
32768	9.8131128921238e-06\\
65536	3.86074203669712e-06\\
};
\addplot [color=red, dashed, line width=0.9pt]
  table[row sep=crcr]{%
64	0.0752499465475519\\
128	0.0300355034124222\\
256	0.0119884665255882\\
512	0.00478511472445327\\
1024	0.00190994593656391\\
2048	0.000762341906235857\\
4096	0.000304283577287464\\
8192	0.000121452716490454\\
16384	4.84770242101342e-05\\
32768	1.93492738917424e-05\\
65536	7.7231308282202e-06\\
};
\addlegendentry{\scriptsize $\mathcal{O}(N^{-1.33})$}

\addplot [color=mycolor3, line width=0.9pt, mark=o, mark options={solid, mycolor3}, forget plot]
  table[row sep=crcr]{%
64	0.00154803166332284\\
128	0.000440585775940739\\
256	0.000123174658678543\\
512	3.45770574289323e-05\\
1024	9.74323691507743e-06\\
2048	2.7117084502335e-06\\
4096	7.58438904347318e-07\\
8192	2.16171884131783e-07\\
16384	6.01086272100687e-08\\
32768	1.68999764389143e-08\\
65536	4.59086816815954e-09\\
};
\addplot [color=mycolor4, line width=0.9pt, mark=triangle, mark options={solid, mycolor4}, forget plot]
  table[row sep=crcr]{%
64	0.00154678921195924\\
128	0.00044006393769157\\
256	0.000122429006985678\\
512	3.45517665647055e-05\\
1024	9.87865886386179e-06\\
2048	2.71347371275297e-06\\
4096	7.54074982760637e-07\\
8192	2.13669729864059e-07\\
16384	6.0070244689273e-08\\
32768	1.65029649618247e-08\\
65536	4.58542657148393e-09\\
};
\addplot [color=darkgray, dashed, line width=0.9pt]
  table[row sep=crcr]{%
64	0.00309606332664568\\
128	0.000867004463582347\\
256	0.000242791138476522\\
512	6.79898886323633e-05\\
1024	1.90395126661024e-05\\
2048	5.33171990504074e-06\\
4096	1.49306537642736e-06\\
8192	4.18109776580458e-07\\
16384	1.17085151147543e-07\\
32768	3.2787878655607e-08\\
65536	9.18173633631909e-09\\
};
\addlegendentry{\scriptsize $\mathcal{O}(N^{-1.84})$}

\addplot [color=mycolor5, line width=0.9pt, mark=o, mark options={solid, mycolor5}, forget plot]
  table[row sep=crcr]{%
64	9.64007051184768e-06\\
128	1.32610885905417e-06\\
256	1.78528593987912e-07\\
512	2.43018972808131e-08\\
1024	3.34862659554665e-09\\
2048	4.52158064317184e-10\\
4096	6.09442619802181e-11\\
8192	8.69880366130298e-12\\
16384	1.17487039916772e-12\\
32768	1.65042792429017e-13\\
65536	2.27184771783602e-14\\
};
\addplot [color=mycolor6, line width=0.9pt, mark=triangle, mark options={solid, mycolor6}, forget plot]
  table[row sep=crcr]{%
64	9.63120981937429e-06\\
128	1.3243470327549e-06\\
256	1.77417418283065e-07\\
512	2.43736888765063e-08\\
1024	3.40110026357102e-09\\
2048	4.47401526916463e-10\\
4096	6.06306431225999e-11\\
8192	8.47943711288475e-12\\
16384	1.16290107298636e-12\\
32768	1.56993422619375e-13\\
65536	2.21601217849075e-14\\
};
\addplot [color=black, dashed, line width=0.9pt]
  table[row sep=crcr]{%
64	1.92801410236954e-05\\
128	2.6444789042055e-06\\
256	3.62718751185126e-07\\
512	4.97507816198007e-08\\
1024	6.82385529751072e-09\\
2048	9.35965217134044e-10\\
4096	1.28377705782294e-10\\
8192	1.76083844145299e-11\\
16384	2.41517948775046e-12\\
32768	3.31267867666342e-13\\
65536	4.54369543567204e-14\\
};
\addlegendentry{\scriptsize $\mathcal{O}(N^{-2.87})$}

\end{axis}
\end{tikzpicture}
		\caption{Weight sequence $\bsgamma=(\gamma_j)_{j=1}^d$ with $\gamma_j = 1/j^2$.}    
	\end{subfigure}
	\begin{subfigure}[b]{0.5\textwidth}  
		\centering 
		\begin{tikzpicture}

\begin{axis}[%
width=0.8\textwidth,
height=0.8\textwidth,
at={(0\textwidth,0\textwidth)},
scale only axis,
xmode=log,
xmin=42.6666666666667,
xmax=98304,
xminorticks=true,
xlabel style={font=\color{white!15!black}},
xlabel={Number of points $N=2^m$},
ymode=log,
ymin=1e-15,
ymax=1,
yminorticks=true,
ylabel style={font=\color{white!15!black}},
ylabel={Worst-case error $e_{N,d,\alpha,\mathbf{\gamma^{\alpha}}}(\mathbf{g})$},
axis background/.style={fill=white},
axis x line*=bottom,
axis y line*=left,
xmajorgrids,
xminorgrids,
ymajorgrids,
yminorgrids,
minor grid style={opacity=0},
legend style={at={(0.03,0.03)}, anchor=south west, legend cell align=left, align=left, draw=white!15!black}
]
\addplot [color=mycolor1, line width=0.9pt, mark=o, mark options={solid, mycolor1}, forget plot]
  table[row sep=crcr]{%
64	0.0126899976166382\\
128	0.00477659788511209\\
256	0.00178815599820943\\
512	0.000668474790747192\\
1024	0.000249237970511458\\
2048	9.28170293056045e-05\\
4096	3.45197186571936e-05\\
8192	1.2873866299598e-05\\
16384	4.78116956633218e-06\\
32768	1.77061615607392e-06\\
65536	6.5556313386119e-07\\
};
\addplot [color=mycolor2, line width=0.9pt, mark=triangle, mark options={solid, mycolor2}, forget plot]
  table[row sep=crcr]{%
64	0.0126883996208652\\
128	0.00477596267110885\\
256	0.00179072880659425\\
512	0.000667928409719759\\
1024	0.000249528553805854\\
2048	9.3065957832874e-05\\
4096	3.45161015000271e-05\\
8192	1.28671774179377e-05\\
16384	4.77544778112109e-06\\
32768	1.76952645794544e-06\\
65536	6.55128760870105e-07\\
};
\addplot [color=red, dashed, line width=0.9pt]
  table[row sep=crcr]{%
64	0.0253799952332763\\
128	0.00945816387379709\\
256	0.00352469978979\\
512	0.00131352224109415\\
1024	0.000489500037094447\\
2048	0.000182418141710241\\
4096	6.79803389240538e-05\\
8192	2.53337000196499e-05\\
16384	9.44091139943581e-06\\
32768	3.51827044540921e-06\\
65536	1.31112626772238e-06\\
};
\addlegendentry{\scriptsize $\mathcal{O}(N^{-1.42})$}

\addplot [color=mycolor3, line width=0.9pt, mark=o, mark options={solid, mycolor3}, forget plot]
  table[row sep=crcr]{%
64	0.000670898194025467\\
128	0.00017515348061943\\
256	4.55119069468944e-05\\
512	1.18108770457534e-05\\
1024	3.05917873296994e-06\\
2048	7.9290021588215e-07\\
4096	2.05145986437275e-07\\
8192	5.33957622522715e-08\\
16384	1.38059638553584e-08\\
32768	3.56823605092433e-09\\
65536	9.23133855954468e-10\\
};
\addplot [color=mycolor4, line width=0.9pt, mark=triangle, mark options={solid, mycolor4}, forget plot]
  table[row sep=crcr]{%
64	0.000670708492223081\\
128	0.000175125591950862\\
256	4.54652696243151e-05\\
512	1.18025820939106e-05\\
1024	3.06345545838972e-06\\
2048	7.94986360971004e-07\\
4096	2.05111977025723e-07\\
8192	5.33005202542887e-08\\
16384	1.376958062069e-08\\
32768	3.56403343791969e-09\\
65536	9.2060582961137e-10\\
};
\addplot [color=darkgray, dashed, line width=0.9pt]
  table[row sep=crcr]{%
64	0.00134179638805093\\
128	0.000347974503699696\\
256	9.02418998168099e-05\\
512	2.34028654282531e-05\\
1024	6.06917752579167e-06\\
2048	1.57394896588627e-06\\
4096	4.08179747039337e-07\\
8192	1.05855214815863e-07\\
16384	2.74519414179377e-08\\
32768	7.11924385515442e-09\\
65536	1.84626771190894e-09\\
};
\addlegendentry{\scriptsize $\mathcal{O}(N^{-1.95})$}

\addplot [color=mycolor5, line width=0.9pt, mark=o, mark options={solid, mycolor5}, forget plot]
  table[row sep=crcr]{%
64	5.54227384997228e-06\\
128	7.02970375585347e-07\\
256	8.90214387577478e-08\\
512	1.12699649048777e-08\\
1024	1.42626908543636e-09\\
2048	1.80563695042224e-10\\
4096	2.28480287388036e-11\\
8192	2.90012290801162e-12\\
16384	3.66945012929552e-13\\
32768	4.64392771694506e-14\\
65536	5.88705903865588e-15\\
};
\addplot [color=mycolor6, line width=0.9pt, mark=triangle, mark options={solid, mycolor6}, forget plot]
  table[row sep=crcr]{%
64	5.54168171547323e-06\\
128	7.02954709134525e-07\\
256	8.89967547930034e-08\\
512	1.12675174304235e-08\\
1024	1.42670185953979e-09\\
2048	1.80786104448765e-10\\
4096	2.28473605289084e-11\\
8192	2.8964272716977e-12\\
16384	3.66433575093592e-13\\
32768	4.64189489401149e-14\\
65536	5.87953689698115e-15\\
};
\addplot [color=black, dashed, line width=0.9pt]
  table[row sep=crcr]{%
64	1.10845476999446e-05\\
128	1.40390919226144e-06\\
256	1.77811587217585e-07\\
512	2.25206592585292e-08\\
1024	2.8523455719348e-09\\
2048	3.61262748498571e-10\\
4096	4.57555966348825e-11\\
8192	5.7951577684527e-12\\
16384	7.33983513082518e-13\\
32768	9.29624039589859e-14\\
65536	1.17741180773118e-14\\
};
\addlegendentry{\scriptsize $\mathcal{O}(N^{-2.98})$}

\end{axis}
\end{tikzpicture}
		\caption{Weight sequence $\bsgamma=(\gamma_j)_{j=1}^d$ with $\gamma_j = 1/j^3$.}
	\end{subfigure}
	\vspace{-16pt}
	\vskip\baselineskip
	\hspace{-0.25cm}
	\centering
	\begin{subfigure}[b]{0.5\textwidth}
		\centering
		\begin{tikzpicture}

\begin{axis}[%
width=0.8\textwidth,
height=0.8\textwidth,
at={(0\textwidth,0\textwidth)},
scale only axis,
xmode=log,
xmin=42.6666666666667,
xmax=98304,
xminorticks=true,
xlabel style={font=\color{white!15!black}},
xlabel={Number of points $N=2^m$},
ymode=log,
ymin=0.0000025,
ymax=1000000000000,
yminorticks=true,
ylabel style={font=\color{white!15!black}},
ylabel={Worst-case error $e_{N,d,\alpha,\mathbf{\gamma^{\alpha}}}(\mathbf{g})$},
axis background/.style={fill=white},
axis x line*=bottom,
axis y line*=left,
xmajorgrids,
xminorgrids,
ymajorgrids,
yminorgrids,
minor grid style={opacity=0},
legend style={at={(0.03,0.03)}, anchor=south west, legend cell align=left, align=left, draw=white!15!black}
]
\addplot [color=mycolor1, line width=0.9pt, mark=o, mark options={solid, mycolor1}, forget plot]
  table[row sep=crcr]{%
64	7243051451.09542\\
128	3621525725.07947\\
256	1810762862.09742\\
512	905381430.634987\\
1024	452690714.944207\\
2048	226345357.140171\\
4096	113172678.280671\\
8192	56586338.8980775\\
16384	28293169.2538933\\
32768	14146584.4657356\\
65536	7073292.10942922\\
};
\addplot [color=mycolor2, line width=0.9pt, mark=triangle, mark options={solid, mycolor2}, forget plot]
  table[row sep=crcr]{%
64	7243051451.09558\\
128	3621525725.08\\
256	1810762862.09309\\
512	905381430.631196\\
1024	452690714.936109\\
2048	226345357.135407\\
4096	113172678.269114\\
8192	56586338.8874833\\
16384	28293169.2380297\\
32768	14146584.4572523\\
65536	7073292.09608587\\
};
\addplot [color=red, dashed, line width=0.9pt]
  table[row sep=crcr]{%
64	14486102902.1908\\
128	7243051417.99143\\
256	3621525692.44372\\
512	1810762837.94586\\
1024	905381414.834932\\
2048	452690705.348467\\
4096	226345351.639734\\
8192	113172675.302617\\
16384	56586337.3926836\\
32768	28293168.5670293\\
65536	14146584.2188584\\
};
\addlegendentry{\scriptsize $\mathcal{O}(N^{-1})$}

\addplot [color=mycolor3, line width=0.9pt, mark=o, mark options={solid, mycolor3}, forget plot]
  table[row sep=crcr]{%
64	10959.7682987585\\
128	5479.51648321991\\
256	2739.45924660214\\
512	1369.47883692103\\
1024	684.551730472485\\
2048	342.13332081745\\
4096	170.958227155933\\
8192	85.4091372596137\\
16384	42.6553497390161\\
32768	21.2937135206517\\
65536	10.6260986771556\\
};
\addplot [color=mycolor4, line width=0.9pt, mark=triangle, mark options={solid, mycolor4}, forget plot]
  table[row sep=crcr]{%
64	10959.7656361317\\
128	5479.51083503195\\
256	2739.44210618333\\
512	1369.45795946705\\
1024	684.527632517571\\
2048	342.109645550625\\
4096	170.941903843961\\
8192	85.3932036462252\\
16384	42.6417266330489\\
32768	21.285663384164\\
65536	10.6181010963113\\
};
\addplot [color=darkgray, dashed, line width=0.9pt]
  table[row sep=crcr]{%
64	21919.536597517\\
128	10951.8784632406\\
256	5471.997153771\\
512	2734.0289568935\\
1024	1366.03037740633\\
2048	682.523492405557\\
4096	341.016075037777\\
8192	170.385290364584\\
16384	85.1313157874084\\
32768	42.5350152714936\\
65536	21.2521973543112\\
};
\addlegendentry{\scriptsize $\mathcal{O}(N^{-1})$}

\addplot [color=mycolor5, line width=0.9pt, mark=o, mark options={solid, mycolor5}, forget plot]
  table[row sep=crcr]{%
64	8.48914060749824\\
128	4.09084925434897\\
256	1.96232776439108\\
512	0.923424072567116\\
1024	0.43274258406824\\
2048	0.200080337268139\\
4096	0.0894074266654823\\
8192	0.0412869601241565\\
16384	0.0182044469434334\\
32768	0.00792331283735132\\
65536	0.00344790766370751\\
};
\addplot [color=mycolor6, line width=0.9pt, mark=triangle, mark options={solid, mycolor6}, forget plot]
  table[row sep=crcr]{%
64	8.46766996208136\\
128	4.0806523512221\\
256	1.93692177499604\\
512	0.906951394555028\\
1024	0.418189445550568\\
2048	0.190921292344443\\
4096	0.0859625367018254\\
8192	0.0379177033253264\\
16384	0.0166350232569472\\
32768	0.00714293789454527\\
65536	0.00304175692103127\\
};
\addplot [color=black, dashed, line width=0.9pt]
  table[row sep=crcr]{%
64	16.9782812149965\\
128	7.77611925973374\\
256	3.56149306139376\\
512	1.63117776395706\\
1024	0.747085801309042\\
2048	0.34216822154539\\
4096	0.156714384921236\\
8192	0.0717758017688485\\
16384	0.0328735981840485\\
32768	0.0150562366554476\\
65536	0.00689581532741503\\
};
\addlegendentry{\scriptsize $\mathcal{O}(N^{-1.13})$}

\end{axis}
\end{tikzpicture}
		\caption{Weight sequence $\bsgamma=(\gamma_j)_{j=1}^d$ with $\gamma_j = (0.95)^j$.}
	\end{subfigure}
	\begin{subfigure}[b]{0.5\textwidth}  
		\centering 
		\begin{tikzpicture}
		\begin{axis}[%
		width=0.8\textwidth,
		height=0.8\textwidth,
		at={(0\textwidth,0\textwidth)},
		scale only axis,
		xmode=log,
		xmin=42.6666666666667,
		xmax=98304,
		xminorticks=true,
		xlabel style={font=\color{white!15!black}},
		xlabel={Number of points $N=2^m$},
		ymode=log,
		ymin=1e-12,
		ymax=1,
		yminorticks=true,
		ylabel style={font=\color{white!15!black}},
		ylabel={Worst-case error $e_{N,d,\alpha,\mathbf{\gamma^{\alpha}}}(\mathbf{g})$},
		axis background/.style={fill=white},
		axis x line*=bottom,
		axis y line*=left,
		xmajorgrids,
		xminorgrids,
		ymajorgrids,
		yminorgrids,
		minor grid style={opacity=0},
		legend style={at={(0.03,0.03)}, anchor=south west, legend cell align=left, align=left, draw=white!15!black}
		]
		\addplot [color=mycolor1, line width=0.9pt, mark=o, mark options={solid, mycolor1}, forget plot]
		  table[row sep=crcr]{%
		64	0.272681567416126\\
		128	0.124050410116249\\
		256	0.0561138472748829\\
		512	0.0255182768803572\\
		1024	0.0113402504908469\\
		2048	0.00502224881264253\\
		4096	0.00221693909672786\\
		8192	0.000974682437453944\\
		16384	0.000428585794683711\\
		32768	0.000187037169109124\\
		65536	8.04522888423455e-05\\
		};
		\addplot [color=mycolor2, line width=0.9pt, mark=triangle, mark options={solid, mycolor2}, forget plot]
		  table[row sep=crcr]{%
		64	0.274712481779285\\
		128	0.124417585308346\\
		256	0.0559772905760882\\
		512	0.0254974002775106\\
		1024	0.0113237967059648\\
		2048	0.00501661471403071\\
		4096	0.00221532998966387\\
		8192	0.000973192798708207\\
		16384	0.000429486397615613\\
		32768	0.000187353189500543\\
		65536	8.07839905190598e-05\\
		};
		\addplot [color=red, dashed, line width=0.9pt]
		  table[row sep=crcr]{%
		64	0.545363134832252\\
		128	0.241921262154038\\
		256	0.107315462568266\\
		512	0.0476047801821899\\
		1024	0.0211173212318125\\
		2048	0.0093675730525565\\
		4096	0.00415542406784002\\
		8192	0.00184333221494031\\
		16384	0.000817696003864883\\
		32768	0.000362727211794674\\
		65536	0.000160904577684691\\
		};
		\addlegendentry{\scriptsize $\mathcal{O}(N^{-1.17})$}
		
		\addplot [color=mycolor3, line width=0.9pt, mark=o, mark options={solid, mycolor3}, forget plot]
		  table[row sep=crcr]{%
		64	0.00922552613029372\\
		128	0.00315102779994744\\
		256	0.00107940206967596\\
		512	0.000392073069501438\\
		1024	0.000129587861373335\\
		2048	4.25897021514503e-05\\
		4096	1.38517285186983e-05\\
		8192	4.64681421290642e-06\\
		16384	1.57007785183912e-06\\
		32768	5.1782408174543e-07\\
		65536	1.60491868732256e-07\\
		};
		\addplot [color=mycolor4, line width=0.9pt, mark=triangle, mark options={solid, mycolor4}, forget plot]
		  table[row sep=crcr]{%
		64	0.00945313308880435\\
		128	0.00316951406710338\\
		256	0.00109688680255328\\
		512	0.000381105276844875\\
		1024	0.000123143143042282\\
		2048	4.29822557594017e-05\\
		4096	1.38276588312289e-05\\
		8192	4.80757770222531e-06\\
		16384	1.54392319759943e-06\\
		32768	5.05565921043526e-07\\
		65536	1.61957971127494e-07\\
		};
		\addplot [color=darkgray, dashed, line width=0.9pt]
		  table[row sep=crcr]{%
		64	0.0184510522605874\\
		128	0.00616690599052078\\
		256	0.00206116859671776\\
		512	0.000688905585819817\\
		1024	0.000230253316943355\\
		2048	7.69577008151646e-05\\
		4096	2.57216173620352e-05\\
		8192	8.59695121750025e-06\\
		16384	2.87336403445476e-06\\
		32768	9.60366142091337e-07\\
		65536	3.20983737464512e-07\\
		};
		\addlegendentry{\scriptsize $\mathcal{O}(N^{-1.58})$}
		
		\addplot [color=mycolor5, line width=0.9pt, mark=o, mark options={solid, mycolor5}, forget plot]
		  table[row sep=crcr]{%
		64	6.48376721823084e-05\\
		128	1.11767657363802e-05\\
		256	2.07776522515426e-06\\
		512	5.09018715878943e-07\\
		1024	8.87804760035051e-08\\
		2048	1.49846164370422e-08\\
		4096	2.47167714218581e-09\\
		8192	4.96668952503188e-10\\
		16384	1.02068159225192e-10\\
		32768	1.93226221652266e-11\\
		65536	3.03614698079325e-12\\
		};
		\addplot [color=mycolor6, line width=0.9pt, mark=triangle, mark options={solid, mycolor6}, forget plot]
		  table[row sep=crcr]{%
		64	6.5512951579925e-05\\
		128	1.08514549896987e-05\\
		256	2.07633554526966e-06\\
		512	4.1358368190811e-07\\
		1024	8.0791551612054e-08\\
		2048	1.40823660331657e-08\\
		4096	2.45320848306968e-09\\
		8192	4.39452153240671e-10\\
		16384	8.4772239825475e-11\\
		32768	1.53147764523675e-11\\
		65536	2.64712334202646e-12\\
		};
		\addplot [color=black, dashed, line width=0.9pt]
		  table[row sep=crcr]{%
		64	0.000129675344364617\\
		128	2.39831899840215e-05\\
		256	4.43564198443429e-06\\
		512	8.20362921995962e-07\\
		1024	1.51724446235168e-07\\
		2048	2.80611263236512e-08\\
		4096	5.18984797829761e-09\\
		8192	9.59851779546646e-10\\
		16384	1.77522625431713e-10\\
		32768	3.28324468545062e-11\\
		65536	6.07229396158648e-12\\
		};
		\addlegendentry{\scriptsize $\mathcal{O}(N^{-2.43})$}
		
		\end{axis}
		\end{tikzpicture}
		\caption{Weight sequence $\bsgamma=(\gamma_j)_{j=1}^d$ with $\gamma_j = (0.7)^j$.}
	\end{subfigure}
	\vskip\baselineskip
	\begin{tikzpicture}
	\hspace{0.05\linewidth}
	\begin{customlegend}[
	legend columns=5,legend style={align=left,draw=none,column sep=1.5ex},
	legend entries={Algorithm \ref{alg:cbc}, standard CBC \qquad, $\alpha=1.5$, $\alpha=2$, $\alpha=3$}
	]
	\addlegendimage{color=gray, mark=o,solid,line width=1.0pt,line legend}
	\addlegendimage{color=gray, mark=triangle,solid,line width=1.0pt}  
	\addlegendimage{color=mycolor-alpha1, only marks,mark=square*, solid, line width=5pt}
	\addlegendimage{color=mycolor-alpha2, only marks,mark=square*, solid, line width=5pt}
	\addlegendimage{color=mycolor-alpha3, only marks,mark=square*, solid, line width=5pt}
	\end{customlegend}
	\end{tikzpicture}
	\caption{Convergence results for the worst-case error $e_{b^m,d,\alpha,\bsgamma^\alpha}(\bsg)$ in the weighted space 
	$W_{d,\bsgamma}^\alpha$ for smoothness parameters $\alpha=1.5,2,3$ with dimension $d=100$. The generating vectors 
	$\bsg$ were constructed via the component-by-component algorithms for polynomial lattice rules with $N=2^m$ points using $K_{N,d,\bsgamma}$ and $e_{N,d,\alpha,\bsgamma^\alpha}$, respectively, as quality functions.}
	\label{fig:cbc}
\end{figure}

The graphs in Figure \ref{fig:cbc} show that the CBC algorithm with smoothness-independent quality function $K_{b^m,d,\bsgamma}$ (Algorithm~\ref{alg:cbc}) constructs good generators of polynomial lattice rules with worst-case errors that are almost identical to those of polynomial lattice rules constructed by the standard CBC algorithm with the worst-case error $e_{b^m,d,\alpha,\bsgamma^\alpha}$ as the quality measure. We stress that while Algorithm \ref{alg:cbc} needs to be run only once (as the quality function is independent of the particular~$\alpha$), the standard CBC construction returns one generating vector for each choice of $\alpha$ individually and is therefore run three times, 
once for each weight sequence considered.

\subsection{Computational complexity}

We illustrate the computational complexity of the component-by-component construction in Algorithm \ref{alg:cbc} which was stated in Proposition \ref{prop:cost-CBC-alg}. To this end, we measure and compare the computation times of implementations of the smoothness-independent fast CBC construction (Algorithm \ref{alg:cbc}) and the standard fast CBC algorithm for polynomial lattice rules with irreducible modulus $p \in \F_2[x]$. For $m,d \in \N$ we use the weight sequence $\bsgamma = (\gamma_j)^d_{j=1}$ with $\gamma_j = j^{-2}$ and note that the chosen weights do not influence the computation times.

\begin{table}[H]
	\captionof{table}{Computation times (in seconds) for constructing the generating vector $\bsg$ of a polynomial lattice rule with $2^m$ points in $d$ dimensions using the fast implementation of Algorithm~\ref{alg:cbc} (\textbf{bold font}) and the standard fast CBC construction (normal font). For the standard fast CBC algorithm we constructed the polynomial lattice rules with smoothness parameter $\alpha=2$.}	
	\label{tab:cbc_times}
	\centering
	\begin{tabular}{p{2cm}p{2cm}p{2cm}p{2cm}p{2cm}p{2cm}} 
		\toprule[1.2pt]
		& $d=50$ & $d=200$ & $d=500$ & $d=1000$ & $d=2000$ \\ 
		\toprule[1.2pt]
		\multirow{2}{6em}{$m=10$}
		& 0.007 & 0.026 & 0.051 & 0.1 & 0.2 \\
		& \textbf{0.007} & \textbf{0.025} & \textbf{0.056} & \textbf{0.106} & \textbf{0.207} \\
		\midrule
		\multirow{2}{6em}{$m=12$}
		& 0.024 & 0.088 & 0.195 & 0.384 & 0.75 \\
		& \textbf{0.025} & \textbf{0.09} & \textbf{0.198} & \textbf{0.389} & \textbf{0.755} \\
		\midrule
		\multirow{2}{6em}{$m=14$}
		& 0.112 & 0.422 & 0.939 & 1.834 & 3.938 \\
		& \textbf{0.115} & \textbf{0.413} & \textbf{0.928} & \textbf{1.834} & \textbf{3.751} \\
		\midrule
		\multirow{2}{6em}{$m=16$}
		& 0.589 & 1.99 & 4.672 & 9.544 & 19.126 \\
		& \textbf{0.589} & \textbf{2.023} & \textbf{4.789} & \textbf{9.787} & \textbf{19.385} \\
		\midrule
		\multirow{2}{6em}{$m=18$}
		& 2.716 & 9.202 & 22.046 & 43.436 & 86.343 \\
		& \textbf{2.751} & \textbf{9.152} & \textbf{21.809} & \textbf{43.001} & \textbf{85.279} \\
		\midrule
		\multirow{2}{6em}{$m=20$}
		& 13.512 & 44.219 & 103.711 & 204.551 & 404.873 \\
		& \textbf{13.794} & \textbf{43.782} & \textbf{103.704} & \textbf{204.414} & \textbf{406.196} \\
		\midrule
	\end{tabular}
\end{table}

In Figure \ref{fig:times_cbc} below we display the measured computation times for the smoothness-independent CBC construction (Algorithm \ref{alg:cbc}) graphically. Since the timings for both considered algorithms are almost identical, we only display the results for Algorithm \ref{alg:cbc} here. \\

The timings displayed in Table \ref{tab:cbc_times} and Figure \ref{fig:cbc} confirm that the computational complexity of both 
algorithms depends on $m$ and $d$ in a similar way and the measured times are in accordance with Proposition \ref{prop:cost-CBC-alg}. 
Additionally, the linear dependence of the construction cost on the dimension~$d$ is well observable. The measured construction time 
for the fast implementation of Algorithm \ref{alg:cbc} is slightly higher than for the standard fast CBC algorithm but in general both algorithms 
can be executed in comparable time.

\begin{figure}[H]
	\centering
	\textbf{Computation times for the fast implementation of Algorithm \ref{alg:cbc}.} \par\medskip 
	\hspace{-0.25cm}
	\centering
	\begin{tikzpicture}
	\begin{axis}[%
	width=0.8\textwidth,
	height=0.8\textwidth,
	at={(0\textwidth,0\textwidth)},
	scale only axis,
	xmin=10,
	xmax=20,
	xlabel style={font=\color{white!15!black}},
	xlabel={$m$},
	ymode=log,
	ymin=0.001,
	ymax=1000,
	yminorticks=true,
	ylabel style={font=\color{white!15!black}},
	ylabel={Computation time in seconds},
	axis background/.style={fill=white},
	xmajorgrids,
	ymajorgrids,
	yminorgrids,
	minor grid style={opacity=0},
	legend style={at={(0.97,0.03)}, anchor=south east, legend cell align=left, align=left, draw=white!15!black}
	]
	
	\addplot [color=mycolor1, line width=1.1pt, mark=o, mark options={solid, mycolor1}]
	  table[row sep=crcr]{%
	10	0.007\\
	11	0.012\\
	12	0.025\\
	13	0.054\\
	14	0.115\\
	15	0.251\\
	16	0.589\\
	17	1.259\\
	18	2.751\\
	19	5.984\\
	20	13.794\\
	};
	\addlegendentry{Algorithm 1 with $d=50$}
	
	\addplot [color=mycolor2-time, line width=1.1pt, mark=o, mark options={solid, mycolor2-time}]
	  table[row sep=crcr]{%
	10	0.207\\
	11	0.381\\
	12	0.755\\
	13	1.665\\
	14	3.751\\
	15	8.078\\
	16	19.385\\
	17	40.96\\
	18	85.279\\
	19	183.638\\
	20	406.196\\
	};
	\addlegendentry{Algorithm 1 with $d=2000$}
	
	\end{axis}
	\end{tikzpicture}
	\caption{Computation times (in seconds) for constructing the generating vector $\bsg$ of a polynomial lattice rule with $2^m$ points in 
		$d \in \{50,2000\}$ dimensions using the fast implementation of Algorithm \ref{alg:cbc}.}
	\label{fig:times_cbc}
\end{figure}
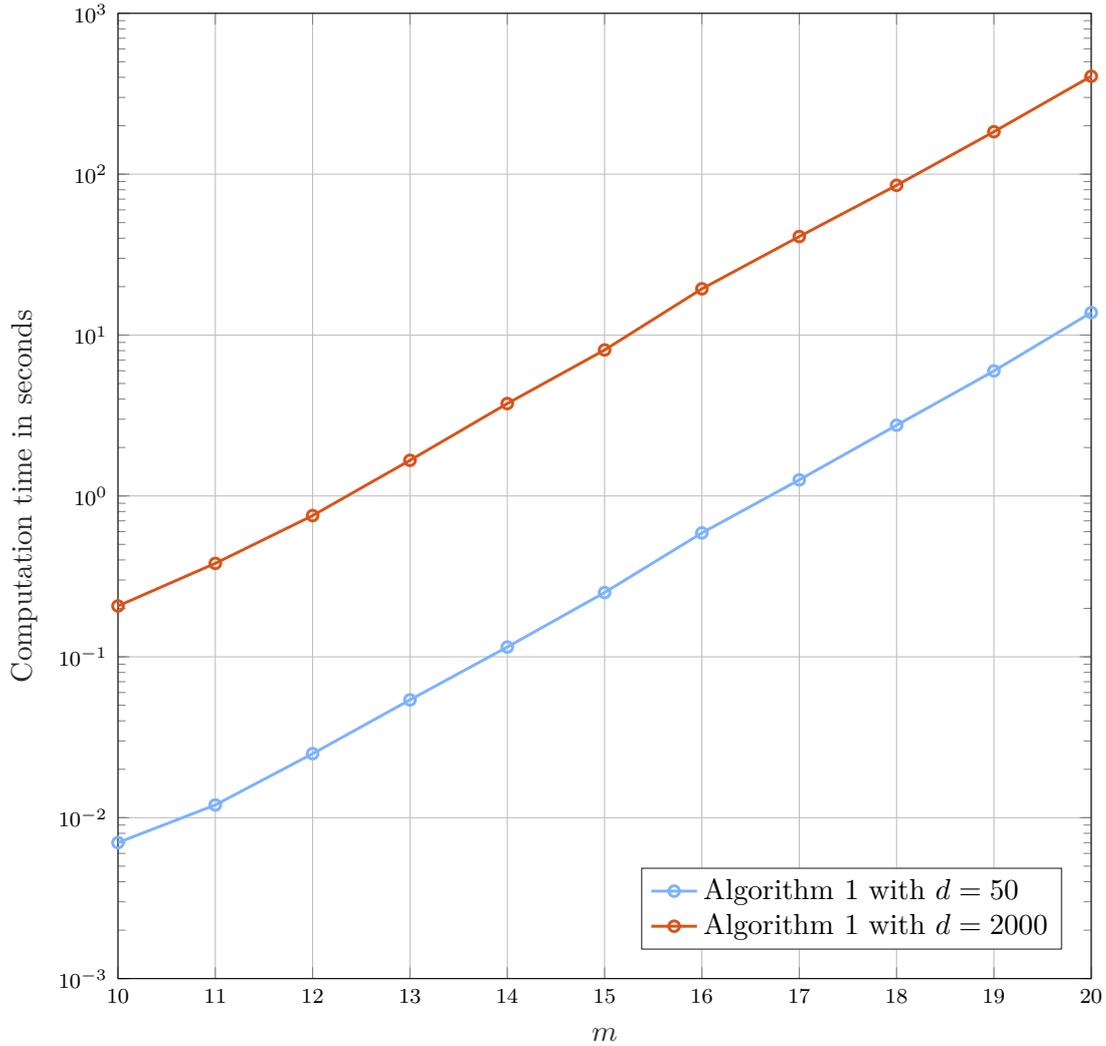

\section{Conclusion}

In this paper we studied the construction of good polynomial lattice rules based on a quality measure which is independent of the smoothness of the underlying function space and can be used as a search criterion for polynomial lattice rules in weighted Walsh spaces. Based on these findings, we presented a component-by-component algorithm for constructing good polynomial lattice rules that is independent of the value of the smoothness parameter $\alpha$ and yields an error convergence rate that is arbitrarily close to the optimal convergence rate in the studied function space. 
Under suitable summability conditions on the weight sequences, these error bounds were shown to be independent of the dimension. Furthermore, we studied a fast implementation of the algorithm which has the same computational complexity as the state-of-the-art standard fast CBC algorithm. 
Numerical experiments with respect to the error behavior and the computational complexity of the algorithm confirmed our theoretical findings.

\bigskip

\bigskip

\begin{small}
	\noindent\textbf{Authors' addresses:}\\
	
	\noindent Adrian Ebert\\
	Johann Radon Institute for Computational and Applied Mathematics (RICAM) \\
	Austrian Academy of Sciences\\
	Altenbergerstr. 69, 4040 Linz, Austria.\\
	\texttt{adrian.ebert@oeaw.ac.at}
	
	\medskip
	
	\noindent Peter Kritzer\\
	Johann Radon Institute for Computational and Applied Mathematics (RICAM) \\
	Austrian Academy of Sciences\\
	Altenbergerstr. 69, 4040 Linz, Austria.\\
	\texttt{peter.kritzer@oeaw.ac.at}
	
	\medskip
	
	\noindent Onyekachi Osisiogu\\
	Johann Radon Institute for Computational and Applied Mathematics (RICAM)\\
	Austrian Academy of Sciences\\
	Altenbergerstr. 69, 4040 Linz, Austria.\\
	\texttt{onyekachi.osisiogu@oeaw.ac.at}
	
	\medskip
	
	\noindent Tetiana Stepaniuk\\
	Institute of Mathematics\\
	University of L\"ubeck\\
	Ratzeburger Allee 160, 23562 L\"ubeck, Germany,\\
	\texttt{stepaniuk@math.uni-luebeck.de}
	
\end{small}

\end{document}